\documentclass{amsart}
\usepackage{mathabx}
\usepackage{comment}
\usepackage{extpfeil}
\usepackage{hyperref}
\usepackage{amssymb, fancyhdr, stmaryrd, amsthm, mathrsfs, textcomp}
\usepackage[utf8]{inputenc}
\usepackage{amsfonts,amsmath,amscd,amssymb}
\usepackage{dsfont,mathtools}
\usepackage{xspace}
\usepackage{biblatex}
\renewbibmacro{in:}{}
\usepackage{xcolor}
\usepackage{tikz-cd}
\usepackage{caption}
\usepackage{subcaption}
\usepackage{graphicx}
\usepackage{listings}
\usepackage{faktor}
\usepackage{calligra}
\usepackage{pifont}

\usepackage{enumitem} 

\def\bC{\ensuremath\mathbb{C}}

\def\bK{\ensuremath\mathbb{K}}
\def\bN{\ensuremath\mathbb{N}}
\def\bQ{\ensuremath\mathbb{Q}}

\def\bZ{\ensuremath\mathbb{Z}}

\def\to{\ensuremath\rightarrow}

\def\<{\ensuremath\langle}
\def\>{\ensuremath\rangle}

\DeclareMathOperator{\Der}{Der}
\DeclareMathOperator{\Irr}{Irr}

\DeclareMathOperator{\im}{im}

\DeclareMathOperator{\Hom}{Hom}

\DeclareMathOperator{\aut}{Aut}
\DeclareMathOperator{\out}{Out}
\DeclareMathOperator{\inn}{Inn}
\DeclareMathOperator{\au}{\mathcal{A}u}
\DeclareMathOperator{\ou}{\mathcal{O}u}

\DeclareMathOperator{\ad}{ad}

\DeclareMathOperator{\ind}{Ind}
\DeclareMathOperator{\res}{Res}

\newtheorem{thm}{Theorem}[section]
\newtheorem{mthm}{Theorem}
\newtheorem{lemma}[thm]{Lemma}
\newtheorem{prop}[thm]{Proposition}
\newtheorem{cor}[thm]{Corollary}

\theoremstyle{definition}
\newtheorem{defn}[thm]{Definition}

\newtheorem{prob}[thm]{Problem}
\newtheorem{ex}[thm]{Example}
\newtheorem*{ex*}{Example}

\newcommand{\sA}{\mathcal{A}}

\newcommand{\sE}{\mathcal{E}{x}}
\newcommand{\sF}{\mathcal{F}}
\newcommand{\sG}{\mathcal{G}{r}}
\newcommand{\sI}{\mathcal{I}{s}}

\newcommand{\sK}{\mathcal{K}}
\newcommand{\sM}{\mathcal{M}{a}}
\newcommand{\sL}{\mathcal{L}}

\newcommand{\sR}{\mathcal{R}}
\newcommand{\sS}{\mathcal{S}{c}}

\newcommand{\mF}{\mathfrak{F}}

\newcommand{\mN}{\mathfrak{N}}

\newcommand{\GRe}{-{\mathrm G}{\mathrm l}{\mathrm R}{\mathrm e}{\mathrm p}}

\newcommand{\Mod}{-{\mathrm M}{\mathrm o}{\mathrm d}}

\newcommand{\GL}{{\mathrm G}{\mathrm L}_n(\bZ)}

\newcommand{\PGLp}{{\mathrm P}{\mathrm G}{\mathrm L}_p(\bK)}
\newcommand{\SLp}{{\mathrm S}{\mathrm L}_p(\bK)}

\begin{document}
\title[Disconnected \& Reductive]{Disconnected Reductive Groups: Classification and Representations}
\author{Dylan Johnston}
\author{Diego Martín Duro}
\author{Dmitriy Rumynin}
\date{\today}
\address{Department of Mathematics, University of Warwick, Coventry, CV4 7AL, UK}


\begin{abstract}
In this article, we classify (up to isomorphism) disconnected reductive groups over an algebraically closed field. For the fixed dimension and number of components, our results show that there are only finitely many reductive groups of that dimension with that many components. Besides this, we obtain new results about the structure and representation theory of disconnected reductive groups.
\end{abstract}

\maketitle


\section*{Introduction}

Reductive groups are fundamental mathematical objects. In this paper, by a reductive group we mean a linear algebraic group $G$ over an algebraically closed field $\bK$ such that the identity component $G_0$ is reductive. We classify reductive groups up to isomorphism and we prove new results about their representations. 

For the convenience of the reader, we summarise our most important results into five abridged theorems, stated in the introduction. A more interested reader should consult the main text. We start with the following finiteness result.

\begin{mthm} (cf. Corollary \ref{prop finite} and Theorem \ref{gr_finite_thm})    
Fix integers $d$ and $n$. Up to isomorphism there are finitely many reductive groups $G$ over $\bK$ with $\dim (G)=d$ and $|\pi_0 (G)|=n$. 
\end{mthm}

Even if one is interested only in connected reductive groups, disconnected ones immediately enter the picture as subgroups. One needs tools to distinguish them up to isomorphism, which we construct in this paper.

Our interest in the subject was triggered by reading a recent paper by Achar, Hardesty and Riche \cite{AHR}, where they prove that the representation ring $R(G)$ does not depend on the characteristic $p$ as soon as $p$ does not divide the order of the component group $\pi_0 (G)$. We have searched through the history of the work on the structure theory of disconnected reductive groups. The earliest theorem on them appears to be the one proved by Borel and Serre \cite{BS}. Some further information is revealed by Spaltenstein \cite{Spa}. It has been refined in the ultimate treatment of the structure theory of disconnected reductive groups by Digne and Michel \cite{DM}. Another trove of valuable information about disconnected reductive groups is the first installation of Lusztig's 10-part paper on their character sheaves \cite{Lus}. 

Late in the work of this paper, we have discovered a recent paper by Gaetz and Vogan \cite{GV}, who proceed on the same tracks as us. They construct all disconnected reductive groups but they do not classify them up to isomorphism. The method of construction is well-known and should be credited to Eilenberg and MacLane \cite{EM}. 

We need to discuss extensions of groups, i.e., short exact sequences
$$ 1 \to N \to G \to H \to 1$$
of abstract groups. One feature of the literature on extensions is that the terminology differs radically between different sources. To remedy this, we follow the terminology in Robinson's textbook \cite{Robinson}, where he talks about couplings, equivalences and isomorphisms of extensions.

A coupling is a homomorphism $\varpi: H \to \out (N)$, obtained by choosing a set-theoretic section of the map $G\to H$. It does not depend on the splitting, so it is the first invariant of the extension.   According to the Eilenberg-MacLane Theorem \cite{EM}, the set $\sE_{\varpi}(N, H)$  of equivalence classes of extensions of $H$ by $N$ with this coupling is a torsor (possibly empty) over $H^2(H,{}_{\varpi}Z)$, where ${}_{\varpi}Z$ is the centre $Z(N)$ with the action of $H$ coming from $\varpi$, cf. Proposition~\ref{EMT}. 

We say that the extensions $G$ and $G'$ (of $H$ by $N$) are {\em{virtually isomorphic}} if there is a group isomorphism $\phi: G\to G'$ that restricts to an automorphism of $N$. 
The group $K\coloneqq \aut (N) \times \aut (H)$ acts on the set of equivalence classes of extensions $\sE(N,H)$ and on the set of couplings, $\hom (H, \out (N))$. Let $K_\varpi$ be the stabiliser of a coupling $\varpi$.

\begin{mthm} (Theorem \ref{thm1})    
The set of virtual isomorphism classes of $H$ by $N$ can be written as a quotient set in two ways:
\[
\sG(N,H)  \ \cong \ 
\dfrac{\sE(N,H)}{K} \ \cong \ 
\coprod_{\varpi\in\hom(H,\out(N))/K} \dfrac{\sE_{\varpi}(N,H)}{K_{\varpi}} \, .
\]
If the $H^2$-torsor $\sE_{\varpi}(N,H)$ admits a $K_\varpi$-stable element, then there is a $K_\varpi$-equivariant bijection 
$\sE_{\varpi}(N,H) \cong H^2 (H,{}_{\varpi} Z)$. If the automorphism exact sequence of $N$
\begin{equation} \label{aut_seq}
 1 \to N/Z \to \aut (N) \to \out (N) \to 1   
\end{equation}
splits, then all $\sE_{\varpi}(N,H)$ admit $K_\varpi$-stable points and we can write
\begin{equation} \label{gr_th2}
\sG(N,H)  \ \cong \ 
\coprod_{\varpi\in\hom(H,\out(N))/A} \dfrac{H^2(H,{}_{\varpi}Z)}{K_{\varpi}} \, .
\end{equation}
\end{mthm}

In our main application, $N$ is a connected reductive group, $\aut (N)$ should contain only automorphisms of algebraic groups --  we exercise certain care in the main text to cover this restriction. The group $H$ is a reduced (as a scheme) finite group because the component group of a linear algebraic group is always finite. Moreover, isomorphisms preserve the identity components and the sequence~\eqref{aut_seq} splits. So the virtual isomorphism classes of extensions are the same isomorphism classes of groups, for which we can use the formula~\eqref{gr_th2}.

We can do a bit better. Let $k$ be the order of the finite group $H$ and $t_k(Z)$ the $k$-torsion of $Z$. As before, $p$ is the characteristic of $\bK$. 

\begin{mthm} (cf. Corollary \ref{groups_g0ss}) 
Let $N$ be a connected reductive group over $\bK$ and $H$ a finite group. The set $\sG(N,H)$ of virtual isomorphism classes of reductive groups $G$ with $G_0 \cong N$ and $\pi_0 (H)\cong H$
can be written as 
\begin{equation} \label{gr_th3}
\sG(N,H)  \ \cong \ 
\coprod_{\varpi\in\hom(H,\out(N))/K} \dfrac{H^2(H,{}_{\varpi}t_k(Z))}{K_{\varpi}} \, .
\end{equation}
This set is finite. If $p$ does not divide $k$, this set is independent of $p$. If $p$ divides $k$, this set is smaller than in characteristic zero.
\end{mthm}

After the classification, we turn our attention to the representation theory of a fixed reductive group $G$ (in any characteristic). Besides its representation ring $\sR(G)$, we need its global representation ring $\sR(G,G_0)$. The latter is the ``gauged'' Grothendieck ring (cf. Definition \ref{global_rep_ring_defn}) of the symmetric monoidal category of pairs $(X,V)$, where $X$ is a finite $G$-set and $V$ is a $G$-equivariant vector bundle on $X$.

Another object of interest is the ``torus'' $T$ of $G$. It is a certain closed subgroup of $G$ such that $\pi_0(T)\rightarrow \pi_0(G)$ is an isomorphism and $T_0$ is a maximal torus in $G_0$.

\begin{mthm} (cf. Theorems \ref{gl_rep_map} and \ref{isom R(G) and R(T)^N}) 
 The global sections functor yields a surjective ring homomorphism
 \[ \sR(G,G_0)\twoheadrightarrow \sR(G)\, . \]
The restriction to the torus $T$ defines a ring isomorphism 
\[ \sR(G)\xrightarrow{\cong} \sR(T)^{N_G(T)} \, .\]
\end{mthm}

Now suppose that the characteristic of $\bK$ is zero. Let $R(G) = \bK [G]$ be the regular $G$-module. For any $G$-module $V$, the $G$-modules $V\otimes R(G)$ and $R(G)^{\oplus \dim(V)}$ are isomorphic. The Knutson Index of $V$ is the smallest positive $n$ such there exists a virtual $G$-module $U$ such that $V\otimes U \cong R(G)^{\oplus n}$. The Knutson Index of $G$ is the least common multiple of the Knutson Indices of its simple modules. 

\begin{mthm} (Theorem \ref{Knutson_finite})       
The Knutson Index of a reductive group $G$ is finite.
\end{mthm}


Let us give some further details on the content of each section. In the first section, we provide technical preliminaries about abstract groups and extensions. The second section is devoted to (selected) structure theory and classifications of disconnected reductive groups. In the third and final section, we study the representation theory and Knutson Index of disconnected reductive groups.

\section*{Acknowledgements}
The authors are indebted to George Lusztig and Jean Michel for sharing their knowledge of the subject. The authors are grateful to Stefan Dawydiak for useful comments that helped to improve the paper.
The first author was supported by the Heilbronn Institute for Mathematical Research (HIMR) and the UK Engineering and Physical Sciences Research Council (EPSRC) under the $\lq\lq$Additional Funding Programme for Mathematical Sciences" grant EP/V521917/1. The second author was supported by the UK Engineering and Physical Sciences Research Council (EPSRC) grant EP/T51794X/1. The third author is grateful to the Max Planck Institute for Mathematics in Bonn for its hospitality and financial support. 

\section{Technical preliminaries} 

Throughout this section, $H$ and $N$ are abstract groups. The centre of $N$ is denoted by $Z$. The group $N$ has additional ``structure" that manifests itself in us fixing a subgroup $\au (N)$ of $\aut (N)$ that contains the group of inner automorphisms $\inn (N)$. The corresponding group of outer automorphisms is $\ou (N) \coloneqq \au (N)/\inn (N)$.

Now, we recap some standard facts, following Robinson \cite{Robinson} with the terminology. Occasionally we introduce new terminology in {\em italic}. The standard facts usually require $\au (N) = \aut (N)$, so we spell everything out carefully to ensure it works in our setup.

\subsection{Groups and extensions}
We consider extensions of $N$ by $H$
\begin{equation} \label{ext1}
\begin{tikzcd}
E \ \ = \ \ \big(    1 \arrow[r] &
    N \arrow[r, "\iota"] & 
    G \arrow[r, "\pi"] &
    H \arrow[r] &
    1 \big)
\end{tikzcd}    
\end{equation}

with an additional assumption that the automorphism 
\begin{equation}
\ad (g) \in \aut (N), \ \ \ 
\ad(g): a \mapsto {}^{\ad(g)}a \coloneqq gag^{-1}
\end{equation}
actually belongs to $\au (N)$ for each $g\in G$.

We say that the two extensions are {\em virtually isomorphic} if there exists a morphism of extensions, i.e., a commutative diagram
\begin{equation} \label{extension_isom}
  \begin{tikzcd}
    1 \arrow[r] &
    N \arrow[r, "\iota"] \arrow[d, "\alpha"] &
    G \arrow[r, "\pi"] \arrow[d, "\psi"] &
    H \arrow[r] \arrow[d, "\beta"] &
    1
    \\
    1 \arrow[r] &
    N \arrow[r, "\iota'"] &
    G' \arrow[r, "\pi'"] &
    H \arrow[r] &
    1
\end{tikzcd}
\end{equation}
with an additional requirement that $\alpha\in \au (N)$ and $\beta\in \aut (H)$.

Recall that an isomorphism of extensions is a virtual isomorphism with $\beta=1$. An equivalence of extensions is an isomorphism with $\alpha =1$. Hereby, we are facing three key sets:
\begin{itemize}
\item $\sG(N,H)$, virtual isomorphism classes of extensions,  
\item $\sI(N,H)$, isomorphism classes of extensions, 
\item $\sE(N,H)$, equivalence classes of  extensions.
\end{itemize}
The sets are connected by surjective functions
\begin{equation} \label{three_sur}
\sF_{eg} : \sE(N,H) \xrightarrow{\sF_{ei}} \sI(N,H) \xrightarrow{\sF_{ig}} \sG(N,H) 
\end{equation}
that can be realised as quotients by group actions. 
Indeed, pick an extension $E$ and two automorphisms $\alpha \in \au (N)$ and $\beta \in \aut (H)$. Consider a new extension
\begin{equation} \label{extension_Eab}
E(\alpha,\beta)
 \ \ = \ \ \big(
  \begin{tikzcd}
    1 \arrow[r] &
    N  \arrow[r, "\iota \circ \alpha^{-1}"] & 
    G \arrow[r, "\beta \circ \pi"] &
    H \arrow[r] &
    1
\end{tikzcd} \big) \, .
\end{equation}
We leave the following proposition without proof.
\begin{prop}
The group $\au(N) \times \aut(H)$ acts on the left on $\sE(N,H)$ by
\[ \au(N) \times \aut(H) \times \sE(N,H) \longrightarrow \sE(N,H),\ \ \ (\alpha,\beta,E) \mapsto E(\alpha,\beta) \]
so that the following statements hold.
\begin{enumerate}
    \item The surjection $\sF_{eg}$ is the quotient map $\bullet /\au(N)\times \aut (H)$.
    \item The surjection $\sF_{ei}$ is the quotient map $\bullet /\au(N)$
    \item The action of $\aut(H)$ descends to an action on $\sI(N,H)=\sE(N,H)/\au(N)$ and the surjection $\sF_{ig}$ is the quotient map $\bullet / \aut (H)$.
\end{enumerate}
\end{prop}

\subsection{Normalised Schreier systems}
Given an extension~\eqref{ext1}, we choose a normalised transversal function $\tau: H\rightarrow G$. Recall that it is a function such that 
$\tau (1_H) = 1_G$ and $\pi\circ \tau = Id_H$. The transversal function yields a bijection
\begin{equation}
G \leftrightarrow N\times H, \ \ \ \ a\cdot \tau(g)\leftrightarrow (a,g)\in N\times H    
\end{equation}
and a normalised Schreier system from $\sS (N,H)$, the set of normalised Schreier systems:
\[
  [\phi, f] \in \sS (N,H), \ \ \phi: H \rightarrow \au (N), \ \ f: H \times H \rightarrow N, 
  \]
defined by identifying $G=N\times H$ and writing down the multiplication as  
  \begin{equation} \label{mult}
    (a,g) \cdot (b,h) = (a \cdot {}^{\phi(g)} b \cdot f(g,h), gh) \, .
  \end{equation}

In the opposite direction, given $[\phi, f] \in \sS (N,H)$, the formula~\eqref{mult}  defines a group structure on the product set $H\times N$.
Denote it $N\#_{\phi,f} H$.
Therefore, we have a surjective function 
\[\sF_{se}: \sS(N,H) \rightarrow \sE(N,H), \ \ \ [\phi,f] \mapsto (N \rightarrow N\#_{\phi,f} H \rightarrow H) .\]
To further the discussion, we need the group homomorphism
\begin{equation} \label{ad_ad}
\widehat{\ad}: {G} \xrightarrow{\ad} \au (N) \rightarrow \ou (N)  
\end{equation}
and the group of normalised functions
  \[ \sM (N, H) = \{ \alpha  : H \rightarrow N \,\mid\, \alpha (1_H)=1_N \} , \ \ \ \alpha\cdot \beta : g \mapsto \alpha (g) \cdot \beta (g) \, . \]
  The group $\sM (N, H)$ acts freely on the set of normalised transversal functions 
\[ \alpha  : (g \mapsto \tau(g)) \mapsto  (g\mapsto \alpha(g) \cdot \tau(g) ) \, . \]
This yields an action of $\sM (N, H)$ on $\sS(N,H)$ where
$\alpha \cdot [\phi,f] = [\phi',f']$ and  
\begin{equation} \label{a_phi_f}
  \phi' (g) = \ad (\alpha (g)) \cdot \phi(g), \quad f'(g,h) =  \alpha (g) \cdot {}^{\phi(g)}\alpha(h) \cdot f(g,h) \cdot \alpha (gh)^{-1}
\end{equation}  
for all $g, h \in H$.

\begin{prop} \label{sc q} 
The following statements hold \cite{Y}.
   \begin{enumerate}[label=(\roman*)] 
   \item The map $\sF_{se}$ is the quotient map under the action~\eqref{a_phi_f}
     of  $\sM (H, N)$.
   \item Once $[\phi,f]\in \sS(H,N)$ is fixed, the centre of $N$ is an $H$-module under $\phi$. We denote it by ${}_{\phi}Z$.
   \item The stabiliser of $[\phi,f]$ in $\sM (H, N)$ is the group of normalised cocycles $Z^{1} (H, {}_{\phi}Z)$.
   \end{enumerate}
 \end{prop}

\subsection{Coupling} The homomorphism from equation ~\eqref{ad_ad} yields a coupling, another group homomorphism, defined by  
\[\varpi =\varpi_G:  H \xrightarrow{\tau} G\xrightarrow{\widehat{\ad}} \ou (N).\]
The coupling does not depend on $\tau$: this justifies the notation $\varpi_G$. By~\eqref{a_phi_f}, equivalent Schreier systems have the same coupling (cf. \cite[Prop. 11.1.1]{Robinson}).
We can say more. 

 \begin{prop} \label{couplings}
 Consider $G,G' \in \sE(N,H)$ with couplings $\varpi,\varpi'$ correspondingly. The following statements hold.
   \begin{enumerate}[label=(\roman*)] 
\item If $\sF_{e,i}(G) = \sF_{e,i}(G')\in \sI(N,H)$, then $\varpi$ and $\varpi'$ are conjugate in $\ou (N)$, that is, there exists $\gamma\in \ou (N)$ such that $\varpi (h)= \gamma \varpi'(h) \gamma^{-1}$ for all $h\in H$. 
    \item If $\sF_{e,g}(G) = \sF_{e,g}(G')\in \sG(N,H)$, then $\varpi$ and $\varpi'$ are equivalent under the $\aut (H)$-action, that is, there exists $\gamma\in \ou (N)$ and $\delta\in\aut (H)$ such that $\varpi ({}^{\delta}h)= \gamma \varpi'(h) \gamma^{-1}$ for all $h\in H$.  
   \end{enumerate}
 \end{prop}
 \begin{proof}
The first statement is essentially \cite[Ex. 11.1(2)]{Robinson}. The second statement admits a straightforward proof. 
\end{proof}

\subsection{Eilenberg-MacLane Theorem} Fix a coupling $\varpi$. Note that it determines an action of $H$ on $Z = Z(N)$. We denote the centre with this action by ${}_{\varpi} Z$.

Eilenberg-MacLane Theorem describes the set $\sE_{\varpi}(N,H)$ of all the extensions with this coupling \cite{EM}. 
\begin{prop} \label{EMT} The coupling defines a cocycle $\widetilde{\varpi}\in Z^3 (H,{}_{\varpi}Z)$.
    \begin{enumerate}[label=(\roman*)] 
   \item If $[\widetilde{\varpi}] \neq 0 \in H^3 (H,{}_{\varpi}Z)$, then $\sE_{\varpi}(N,H)=\emptyset$.
   \item If $[\widetilde{\varpi}] = 0 \in H^3 (H,{}_{\varpi}Z)$, then $\sE_{\varpi}(N,H)$ has a free transitive action of $H^2 (H,{}_{\varpi}Z)$.
   \item The set $\sE(N,H)$ is a disjoint union of $\sE_{\varpi}(N,H)$ over all $\varpi\in\hom(H,\ou(N))$.
   \end{enumerate}   
\end{prop}

It is now natural to sort out the (virtual) isomorphism classes into different subsets with fixed couplings. Our notations  $\sI_{\varpi}(N,H)$  and $\sG_{\varpi}(N,H)$ are self-explanatory. The phenomena described in Proposition~\ref{couplings} can be described as a left action of the product group $\ou (N) \times \aut (H)$ on the set $\hom (H, \ou (N))$ of all the couplings:
\begin{equation}
\varpi' = (\gamma,\delta) \star \varpi \ \mbox{ if } \ \forall h\in H \ 
\varpi' (h)= \gamma \varpi({}^{\delta^{-1}}h) \gamma^{-1}.    
\end{equation}

\begin{lemma} \label{change_coupling}
Let us consider an extension
\[E=(N \rightarrow N\#_{\phi,f} H \rightarrow H)\in  \sE_{\varpi}(N,H)\, .\]
\begin{enumerate}
    \item If a coupling $\varpi'= (\gamma,1) \star \varpi$ is conjugate to $\varpi$, then $E$ is isomorphic to an extension in  $\sE_{\varpi'}(N,H)$.
    \item If a coupling $\varpi'= (\gamma,\delta) \star \varpi$ is equivalent to $\varpi$, then $E$ is virtually isomorphic to an extension in  $\sE_{\varpi'}(N,H)$. 
\end{enumerate}
\end{lemma}
\begin{proof} We only prove the second statement. The proof of the first one is identical, once we let $\delta=1$.

For convenience, let us translate the slightly different relation $\varpi'=(\gamma,\delta^{-1}) \star \varpi$ into a statement about $\au (N)$. The couplings lift to normalised functions $\phi,\phi' : H \rightarrow \au (N)$, while the element $\gamma$ lifts to $\widehat{\gamma} \in \au (N)$. Finally, we know that there exists another normalised function $\beta: H \rightarrow N$ such that
\[
\forall h\in H \quad  \widehat{\gamma} \cdot \phi'(h) =  \ad (\beta(h)) \cdot \phi ({}^{\delta}h) \cdot  \, \widehat{\gamma} \in \au (N). 
\]
For all $n\in N$ and $h\in H$ consider elements
\[
\widetilde{n} \coloneqq ( {}^{\widehat{\gamma}}n, 1), \;
\widetilde{h} \coloneqq ( \beta (h), {}^{\delta} h) \in 
N\#_{\phi,f} H \, .
\]
Observe that 
\begin{align*}
\widetilde{h} \cdot \widetilde{n}   = ( \beta (h), {}^{\delta}h) \cdot ( {}^{\widehat{\gamma}}n, 1)  \, & = ( \beta (h) \, \cdot \, {}^{\phi ({}^{\delta}h)\widehat{\gamma}} n, {}^{\delta}h)=\\
( {}^{\ad(\beta(h))\phi ({}^{\delta}h)\widehat{\gamma}} n \cdot \beta(h), {}^{\delta}h) = &  \, 
( {}^{\widehat{\gamma} \phi' (h)} n, 1) \cdot (\beta(h), {}^{\delta}h) = 
\widetilde{ {}^{\phi' (h)} n} \cdot \widetilde{h} \, .
\end{align*} 
It follows that the map 
\[
N\#_{\phi,f} H \rightarrow N\#_{\phi,f} H \, , \quad
(n,h) \mapsto \widetilde{n} \widetilde{h} = ( {}^{\widehat{\gamma}}n \cdot \beta (h), {}^{\delta} h)
\]
defines a virtual isomorphism of extensions from $N\#_{\phi,f} H$ to $N\#_{\phi',f'} H$ for some appropriately defined $f'$. 
\end{proof}

%

Let us examine the Eilenberg-MacLane bijection $\sE_{\varpi}(N,H) \leftrightarrow H^2 (H,{}_{\varpi}Z)$
(cf. Proposition~\ref{EMT}). Choose any $E \in \sE_{\varpi}(N,H)$, represent it by a Schreier system $E=\sF_{se} [\phi,f]$ for some $[\phi,f] \in \sS (N,H)$. Now we can write both the  $H^2 (H,{}_{\varpi}Z)$-torsor structure on $\sE_{\varpi}(N,H)$ and the bijection explicitly:
\begin{align}
&\Upsilon : H^2 (H,{}_{\varpi}Z) \times \sE_{\varpi}(N,H) \to \sE_{\varpi}(N,H), \ \ ([\tau],E) \mapsto \sF_{se} [\phi, \tau \cdot f] ,\label{EMu}
\\
&\Upsilon_E : H^2 (H,{}_{\varpi}Z) \to \sE_{\varpi}(N,H), \ \ [\tau] \mapsto \Upsilon ([\tau],E)=\sF_{se} [\phi, \tau \cdot f] . \label{EMue}
\end{align} 
Our ultimate aim is to understand the action of the group\footnote{This group is called $K$ in the introduction.} $\au(N)\times \aut (H)$ on $\sE(N, H)$, at least in some cases. Some elements of this group change the coupling, but this can be controlled using Lemma~\ref{change_coupling}. Consider the coupling stabilisers inside this group. We can see that
\begin{equation} \label{inclusions}
(\au(N)\times \aut (H))_{\varpi} \supseteq \au(N)_{\varpi}\times \aut (H)_{\varpi} \supseteq \ \inn (N) \times \inn (H)
\end{equation}
but not much more in this generality.
The stabilisers act by conjugation on both  $H^2 (H,{}_{\varpi}Z)$ and $\sE_{\varpi}(N,H)$. The next lemma follows directly from formulas~\eqref{EMu} and \eqref{EMue}.
\begin{lemma} \label{eq_lem} The following statements hold.
 \begin{enumerate}
     \item The maps $\Upsilon$ and $\Upsilon_E$ do not depend on the choice of the Schreier system representing $E$.
     \item  The action map $\Upsilon$ is $(\au(N)\times \aut (H))_{\varpi}$-equivariant.
     \item The bijection $\Upsilon_E$ is $\au(N)_{\varpi}$-equivariant if and only if $E$ is a fixed point of the action of $\au(N)_{\varpi}$. 
     \item The bijection $\Upsilon_E$ is $(\au(N)\times \aut (H))_{\varpi}$-equivariant if and only if $E$ is a fixed point of the action of $(\au(N)\times \aut (H))_{\varpi}$.
 \end{enumerate}   
\end{lemma}

Notice that $[0] \in H^2 (H,{}_{\varpi}Z)$ is fixed but existence of a fixed point in $\sE_{\varpi}(N,H)$ 
cannot be guaranteed, in general. A good algebraic example to keep in mind is a group $G$, acting on $X=G$ by left multiplication. Another copy of $G$ acts on $G$ by conjugation. Then the action map $G\times X \to X$ is $G$-equivariant but $X$ has no fixed point, so no equivariant bijection $G\leftrightarrow X$ exists.

Let us summarize what we know about (virtual) isomorphism classes of extensions at this point. 

\begin{thm} \label{thm1}
Let $\hom_0(H,\ou(N))$ be the set of all such couplings $\varpi$ that $[\widetilde{\varpi}] = 0$. The following statements hold. 
    \begin{enumerate}[label=(\roman*)] 
   \item The set $\hom_0(H,\ou(N))$ is closed under the action of $\au(N)\times \aut (H)$ on $\hom(H,\ou(N))$.
   \item The sets of (virtual) isomorphism classes can be described as disjoint unions over representatives of the orbits as
\begin{align*}
\sI(N,H) &= \coprod_{\varpi \in \hom_0(H,\ou(N))/\au (N)} \dfrac{\sE_{\varpi}(N,H)}{\au (N)_{\varpi}}  , 
\\
\sG(N,H) &= \coprod_{\varpi \in \hom_0(H,\ou(N))/(\au (N)\times \aut (H))} \dfrac{\sE_{\varpi}(N,H)}{(\au (N) \times \aut (H))_{\varpi}} . 
\end{align*} 
   \end{enumerate}      
\end{thm}

\begin{proof}
A coupling $\varpi$ belongs to $\hom_0(H,\ou(N))$ if and only if there exists an extension with this coupling. The action of $\au(N)\times \aut (H)$ ``moves'' extensions around. The first statement follows.

Thanks to Lemma~\ref{change_coupling}, we can write
\[
\sI(N,H) = \coprod \sI_{\varpi}(N,H), \ \ 
\sG(N,H) = \coprod \sG_{\varpi}(N,H)
\]
where the disjoint unions are taken over the same representatives of the orbits as in the second statement. The description of $\sI_{\varpi}(N, H)$ and $\sG_{\varpi}(N, H)$ as quotients of $\sE_{\varpi}(N, H)$ immediately follows from the definitions of (virtual) isomorphism.
\end{proof}

The quotient sets in the statement of Theorem~\ref{thm1} are $\sI_{\varpi}(N,H)$ and $\sG_{\varpi}(N,H)$. Thanks to Lemma~\ref{eq_lem}, once the existence of an extension, fixed under the groups, is established, this can be rewritten as 
\begin{equation}
  \sI_{\varpi}(N,H) = \dfrac{H^2 (H,{}_{\varpi}Z)}{\au (N)_{\varpi}}, \ \ 
\sG_{\varpi}(N,H) = \dfrac{H^2 (H,{}_{\varpi}Z)}{ (\au (N)\times \aut (H))_{\varpi}}.   
\end{equation}


\subsection{Semidirect products} 
A semidirect product $N \rtimes_{\phi} H$ is a crossed product $N\#_{\phi,1} H$ with the function $f$ being $1$. This condition forces the second function $\phi: H \rightarrow \au (N)$ to be a homomorphism.

Consider the associated set of derivations
\[ \Der_\phi (H, N) \coloneqq \{ \alpha : H \rightarrow N \, \mid \, \forall h_1,h_2 \in H \ 
\alpha (h_1) \cdot {}^{\phi (h_1)} \alpha (h_2) = \alpha (h_1 \cdot h_2)\} .\]
Their role is ``to compare'' different semidirect products.
\begin{prop} \label{semidirpro}
Consider a semidirect product $E = N \rtimes_{\phi} H$.
\begin{enumerate}
    \item The derivations are in bijection with the group complements of $N$ in $E$:
    \[
\Der_\phi (H, N) \ni \alpha \ \longleftrightarrow \ \{ (\alpha (h),h) \, | \, h\in H \}.
    \]
    \item Consider a function $\beta: H\to N$. Define a new function
    \[\psi \coloneqq \ad (\beta)\cdot \phi : H \rightarrow \au (N) \, ,\]
    which, on elements, can be written as
    \[ \forall h \in H \  \forall n \in N \quad {}^{\psi(h)}n = \beta (h)\cdot {}^{\phi(h)}n \cdot \beta (h)^{-1} \, .
    \]
    The function $\psi$ is a homomorphism if and only if the composition
    \[
    \overline{\beta} : H \xrightarrow{\beta} N \twoheadrightarrow N/Z =\inn (N)
    \]
    is a derivation, i.e., belongs to $\Der_{\, \overline{\phi}\, } (H, \inn (N))$. 
        \item Let $E' = N \rtimes_{\psi} H$ be another extension. The extensions $E$ and $E'$ are equivalent if and only if there exists $\alpha \in \Der_\phi (H, N)$ such that 
    \[\psi = \ad (\alpha)\cdot \phi \, .\]
\end{enumerate}
\end{prop}
\begin{proof}
The first statement is \cite[Prop. 11.1.2]{Robinson}.  To prove the second statement, we need the two calculations:
\[
{}^{\psi(h_1h_2)}n = \beta (h_1 h_2)\cdot {}^{\phi(h_1 h_2 )}n \cdot \beta (h_1h_2)^{-1} = {}^{\overline{\beta} (h_1h_2)} ({}^{\phi(h_1h_2)}n) \, , 
\]
\[
{}^{\psi(h_1)}({}^{\psi(h_2)}n) = {}^{\overline{\beta} (h_1)\phi(h_1)} ( \beta (h_2) \cdot {}^{\phi(h_2)}n \beta (h_2)^{-1}) = {}^{\overline{\beta}(h_1) \cdot {}^{\phi(h_1)}\overline{\beta} (h_2)} ({}^{\phi(h_1h_2)}n) \, .  
\]
Since ${}^{\phi(h_1h_2)}n$ could be any element of $N$, $\psi$ being a homomorphism is equivalent to $\overline{\beta}$ being a derivation.

To prove the third statement, observe that an equivalence of extensions must be of the form 
\[
    N \rtimes_{\psi} H \rightarrow N \rtimes_{\phi} H , \quad 
    (n,h) \mapsto (n \cdot \alpha(h),h)    
\]
for some $\alpha: N \rightarrow H$ such that $\alpha(1)=1$. 
Let $(n_1,h_1), (n_2,h_2) \in E'$. First multiplying in $E'$ and then applying the above map yields
\[ \Big( n_1 \cdot {}^{\psi(h_1)}n_2 \cdot \alpha(h_1h_2) , h_1h_2\Big).\]
On the other hand, applying the above map and then multiplying in $E$ yields
\[ \Big(n_1 \cdot \alpha(h_1) \cdot {}^{\phi(h_1)}\big(n_2 \cdot \alpha(h_2)\big), h_1h_2 \Big).\]
The equality of these two expressions is equivalent to the condition:
\begin{equation} \label{equat_hnn}
 {}^{\psi(h_1)}n \cdot \alpha(h_1h_2) = \alpha(h_1) \cdot {}^{\phi(h_1)}\big(n \cdot \alpha(h_2)\big) \, .   
\end{equation}
For $n=1$, this condition reads as $\alpha \in \Der_\phi (H, N)$.
For $h_2=1$, it reads as $\psi = \ad (\alpha)\cdot \phi$. 
In the opposite direction, we assume that $\alpha \in \Der_\phi (H, N)$ and $\psi = \ad (\alpha)\cdot \phi$, then we can easily transform the left hand side of \eqref{equat_hnn} into the right hand side. 
\end{proof}

Note that $\phi$ and $\psi = \ad (\beta) \cdot \phi$ in part (2) of Proposition~\ref{semidirpro} have the same coupling. In the opposite direction, $\phi$ and $\psi$ have the same coupling if and only if $\psi = \ad (\beta) \cdot \phi$. This makes the classification of semidirect products inside $\sE_{\varpi} (N, H)$ a somewhat subtle problem, as explained in the next corollary.  
\begin{cor}
Consider a homomorphism $\phi: H \rightarrow \au (N)$ with the coupling $\varpi$. Then $\sE_{\varpi} (N, H)$ has a unique semidirect product if and only if the natural map  $\Der_{\phi} (H, N)\rightarrow\Der_{\, \overline{\phi}\, } (H, \inn (N))$ is surjective    
\end{cor}

Of course, this map is not always surjective: there is an exact sequence
\[
\Der_{\phi} (H,N)\rightarrow\Der_{\, \overline{\phi}\, } (H, \inn (N)) \rightarrow H^2 (H, {}_{\phi}Z) 
\]
that controls the lack of surjectivity. We will not use it, so supply no proof. We finish this section with the computation of the action of $\au(N)\times \aut(H)$ on a semidirect product.
\begin{lemma} \label{action_on_sd}
Consider a semidirect product $E = N \rtimes_{\phi} H$ and a pair of automorphisms $(\alpha,\beta)\in\au(N)\times \aut(H)$. Then, as elements of $\sE(N,H)$,
\[
E (\alpha, \beta) = N \rtimes_{\psi} H \quad \mbox{ where }  \quad \psi (h) = \alpha \cdot \phi (\beta^{-1}(h)) \cdot \alpha^{-1}\, . 
\]
\end{lemma}
\begin{proof} Using the canonical identification $E=E(\alpha,\beta)$ as sets, the proof boils down to finding commutation relations between
\[
\widetilde{n} = (\alpha^{-1} (n), 1), \; \widetilde{h} = (1, \beta^{-1} (h)) \in N \rtimes_{\phi} H
\]
for arbitrary $n\in N$, $h\in H$. The computation
\begin{align*}
\widetilde{h} \cdot \widetilde{n} = (1, \beta^{-1} (h)) \cdot (\alpha^{-1} (n), 1) &
= ({}^{\phi(\beta^{-1} (h))} \alpha^{-1}(n)), \beta^{-1} (h)) = \\
(\alpha^{-1} ({}^{\alpha \cdot \phi(\beta^{-1} (h)) \cdot  \alpha^{-1}}n), \beta^{-1} (h)) & 
= (\alpha^{-1} ({}^{\psi(h)} n), \beta^{-1} (h)) = \widetilde{{}^{\psi(h)}n} \cdot \widetilde{h}
\end{align*}
finishes the proof.
\end{proof}

\subsection{Split automorphisms}

A stronger result can be achieved under the additional assumption that $\au (N)$ is a semidirect product
\begin{equation} \label{split_aut}
 \au (N) = \inn (N) \rtimes \ou (N)   \, .
\end{equation}
While such a semidirect product decomposition is not unique, we fix it in the next theorem. It is essentially a version of Eilenberg-MacLane Theorem (Proposition~\ref{EMT}) for (virtual) isomorphism classes. See Theorem~\ref{thm1} for the notation.
\begin{thm} \label{thm2}
Suppose the automorphisms of the group $N$ are split as in~\eqref{split_aut}.
 Then the following statements hold. 
    \begin{enumerate}[label=(\roman*)] 
   \item Every coupling $\varpi$ admits a canonical lifting
   \[ \widehat{\varpi} : H \xrightarrow{\varpi} \ou (N) \hookrightarrow \au (N) \, .\]
   \item In particular, $\hom_0(H,\ou(N)) = \hom(H,\ou(N))$.
   \item The semidirect product $N \rtimes_{\widehat{\varpi}} H$ is a fixed point in $\sE_{\varpi}(N,H)$ under the action of the group $(\au (N) \times \aut (H))_{\varpi}$.
   \item The sets of (virtual) isomorphism classes can be described as disjoint unions over representatives of the orbits as
\begin{align*}
\sI(N,H) &= \coprod_{\varpi \in \hom(H,\ou(N))/\au (N)} \dfrac{H^2 (H, {}_{\varpi}Z)}{\au (N)_{\varpi}}  , 
\\
\sG(N,H) &= \coprod_{\varpi \in \hom(H,\ou(N))/(\au (N)\times \aut (H))} \dfrac{H^2 (H,  {}_{\varpi}Z)}{(\au (N) \times \aut (H))_{\varpi}} . 
\end{align*} 
   \end{enumerate}      
\end{thm}
\begin{proof}
Statement (i) is obvious and it immediately implies Statement (ii).

To prove Statement (iii), pick $(\alpha,\beta)\in (\au (N) \times \aut (H))_{\varpi}$. Let $\phi=\widehat{\varpi}$. In light of the decomposition~\eqref{split_aut}, we can write
\[
\alpha = (\alpha_1,\alpha_2), \; \phi(h) = (1,\varpi (h)) \in \inn (N) \rtimes \ou (N) \, .
\]
Let $E=N\rtimes_{\phi}H$. By Lemma~\ref{action_on_sd}, $E(\alpha,\beta)=N\rtimes_{\psi}H$ where
\begin{align*}
\psi(h) = (\alpha_1,\alpha_2) \cdot (1, \varpi(\beta^{-1} (h))) \cdot (1,\alpha_2^{-1}) \cdot (\alpha_1^{-1}, 1) &
= (\alpha_1,\alpha_2 \varpi(\beta^{-1} (h))  \alpha_2^{-1}) \cdot (\alpha_1^{-1}, 1)
\\
= (\alpha_1 \cdot {}^{\alpha_2 \varpi(\beta^{-1} (h))  \alpha_2^{-1}}(\alpha_1^{-1}),\alpha_2 \varpi(\beta^{-1} (h))  \alpha_2^{-1}) \stackrel{\ast}{=}
& 
\; (\alpha_1 \cdot {}^{\varpi(h)}(\alpha_1^{-1}), \varpi(h)) \, .
\end{align*}
The last equality, marked by $\ast$, follows from $(\alpha,\beta)$ preserving the coupling. The function that appears in the first component
\[
H \rightarrow \inn (N), \quad h \mapsto \alpha_1 \cdot {}^{\varpi(h)}(\alpha_1^{-1})
\]
belongs to $\Der_{\, \overline{\phi}\, } (H, \inn (N))$. Moreover, it is an inner derivation, so it can be easily lifted. Let $\alpha_1 = \ad (t)$ for some $t\in N$. Then
\[
\gamma \coloneqq \Big( h \mapsto t \cdot {}^{\varpi(h)}(t^{-1}) \Big) \in \Der_\phi (H, N)
\]
is the lift such that $\psi = \gamma \cdot \phi$. Now Statement~(iii) follows from Proposition~\ref{semidirpro}. 

Finally, Statement (iv) easily follows from the first three statements, Lemma~\ref{eq_lem} and Theorem~\ref{thm1}.
\end{proof}

\section{Reductive groups}

Throughout this section, $H$ is a finite group, while $N$ is the group of points of a connected reductive group over an algebraically closed field $\bK$ of characteristic $p\geq 0$. Consequently, $\au (N)$ is the group of automorphisms of $N$ as an algebraic group.

Later, we will also have $G$, the group of $\bK$-points of a reductive group, which is not necessarily connected. The group $G$ is related to $N$ and $H$ by $G_0\cong N$ and $\pi_0 (G)=H$. In other words, it is an extension in $\sE(N,H)$.

\subsection{Classification of connected reductive groups}

To begin, we briefly recall the notion of a root datum and the classification of connected reductive groups over an algebraically closed field. For more details on this see, for example, \cite{Sp}.

Let $X$ and $X^\vee$ be free abelian groups of finite rank, with a fixed pairing $\langle -,- \rangle: X \times X^\vee \rightarrow \mathbb{Z}$. Furthermore, let $\Phi$ and $\Phi^\vee$ be finite subsets of $X$ and $X^\vee$, respectively, with a fixed bijection $\Phi \rightarrow \Phi^\vee; \alpha \mapsto \alpha^\vee$. For each $\alpha \in \Phi$, define endomorphisms $s_\alpha \in \text{End}(X)$, $s_{\alpha^\vee} \in \text{End}(X^\vee)$ via $s_\alpha(x) = x - \langle x, \alpha^\vee \rangle\alpha$ for all $x \in X$ and $s_{\alpha^\vee}(y) = y - \langle \alpha, y \rangle\alpha^\vee$ for all $y \in X^\vee$. Then, a \textit{root datum} is a quadruple $\Psi = (X,\Phi, X^\vee, \Phi^\vee)$ of objects as above that satisfy the following two conditions:
\begin{itemize}
    \item For all $\alpha \in \Phi$, $\langle \alpha, \alpha^\vee \rangle = 2$ .
    \item For all $\alpha \in \Phi$, $s_\alpha(\Phi) \subset \Phi$ and $s_{\alpha^\vee}(\Phi^\vee) \subset \Phi^\vee$.
\end{itemize}
If $\Phi$ is non-empty, then we call it the \textit{root system} of $\Psi$. Note that this is a root system in the conventional sense \cite[Section 1.3]{Sp}. In this case, we also call $\Phi^\vee$ the \textit{coroot system}. Additionally, if for all $\alpha \in \Phi$, the only multiples of $\alpha$ in $\Phi$ are $\pm \alpha$, then we say that $\Phi$ is  \textit{reduced}.

Let $N$ be a connected reductive group over an algebraically closed field. Fix a maximal torus $T$ and denote by $X:= X(T) = \Hom(T,\mathbb{G}_m)$ the group of rational characters of $T$. Furthermore, let $X^\vee = \Hom(\mathbb{G}_m,T)$ be the group of cocharacters. For any $\alpha \in X$, $\beta^\vee \in X^\vee$ the composition $\alpha \circ \beta^\vee \in \Hom(\mathbb{G}_m,\mathbb{G}_m)$ defines a pairing $\langle -,- \rangle$.
Let $\Phi := \Phi(N,T)$ be the root system of $(N,T)$. We define the coroots $\Phi^\vee$ as follows: for each $\alpha \in \Phi$, let $T_\alpha \subset T$ be the connected component of $\ker(\alpha)$. Then, there exists a unique $\alpha^\vee \in X^\vee(T)$ such that $\langle \alpha, \alpha^\vee \rangle = 2$ and $T = \im(\alpha^\vee)T_\alpha$. We define $\Phi^\vee = \{\alpha^\vee\}_{\alpha \in \Phi}$.

In this way, one constructs a (unique) root datum from a connected reductive algebraic group. Write $\psi(N,T)$ for this root datum. In fact, the converse is also true.

\begin{thm} \label{class_redu} \cite[Theorem 2.9]{Sp} Let $\Psi = (X,\Phi, X^\vee, \Phi^\vee)$ be a root datum with a reduced root system $\Phi$. Then there exists a connected reductive group $N$ and a maximal torus $T \subset N$ such that $\Psi = \psi(N, T),$ the root datum of $(N, T)$. Moreover, the pair $(N, T)$ is unique up to isomorphism.  
\end{thm}

Moreover, for each dimension, the number of such groups is always finite. The following fact is well-known but we could not find a proof in the literature.

\begin{cor} \label{prop finite}
Fix a dimension $n$. Up to isomorphism, there are finitely many connected reductive groups over $\bK$ of dimension $n$. This number does not depend on the field $\bK$ (as soon as $\bK$ is algebraically closed).    
\end{cor}
\begin{proof}
By Theorem~\ref{class_redu}
the classification does not depend on an (algebraically closed) field.

We can assume that $\bK$ has characteristic 0.
Every connected reductive group $N$ of dimension $n$ is isomorphic to $N_{sc}/A$, where $N_{sc} = K \times T_m$ is a product of
a simply connected semisimple group $K=[N_{sc},N_{sc}]$ and
an $m$-dimensional torus $T_m$
by a finite subgroup $A$ of $Z(N_{sc}) = Z(K) \times T_m$. 

Observe that if we have $A < T_m < Z(N_{sc})$ then $N_{sc}/A \cong N_{sc}$. More generally, given $A < Z(N_{sc})$, there exists $A' < Z(K) \times t_{|A|}(T_m) \subset N_{sc}$, where $t_{|A|}(T_m)$ denotes the $|A|$-torsion in $T_m$, such that $N_{sc}/A \cong N_{sc}/A'$. However, there are finitely many subgroups of $Z(K) \times t_{|A|}(T_m)$ since it is finite. Therefore, there are finitely many connected algebraic groups of dimension $n$.
\end{proof}

Note that 
different algebraic groups may be isomorphic as abstract groups.
For example, the quotient map $\SLp \to \PGLp$ in characteristic $p$ is an isomorphism of abstract groups between non-isomorphic algebraic groups. \\

Let $\Psi = (X,\Phi, X^\vee, \Phi^\vee)$ be a root datum, and $\Delta \subset \Phi$ is a basis of $\Phi$, then $\Delta^\vee = \{\alpha^\vee : \alpha \in \Delta\}$ is a basis of $\Phi^\vee$. A \textit{based root datum} is a sextuple $\Psi_0 = (X,\Phi,\Delta, X^\vee, \Phi^\vee,\Delta^\vee)$, moreover since $\Delta$ and $\Delta^\vee$ determine $\Phi$ and $\Phi^\vee$ completely, we may also view the based root datum as the quadruple $\Psi_0 = (X,\Delta, X^\vee, \Delta^\vee)$. It follows that if $\Psi_0$ is a based root datum, then there exists a triple $(N, B, T)$, where $B$ is a Borel subgroup of $N$ containing $T$, such that $\psi_0(N, B, T) = \Psi_0$. Let $\alpha \in \Delta$ be a simple root. Associated to each $\alpha$ is a subgroup $N_\alpha \subset N$ isomorphic to $SL_2$, we define $U_\alpha = \text{Im}(x_\alpha) \subset N_\alpha$ where $x_\alpha: \mathbb{G}_a \to N_\alpha$ is the unique homomorphism of algebraic groups satisfying $t x_\alpha(k) t^{-1} = x_\alpha(t^\alpha k)$ for $t \in T$ and $k \in \mathbb{K}$. For each $\alpha \in \Delta$ we choose a $u_\alpha \in U_\alpha$.

\begin{prop} \cite[Proposition 2.13 and Corollary 2.14]{Sp}
$\aut\big(\psi_0(N,B,T)\big)$ is isomorphic to the group $\aut(N, B, T, \{u_\alpha\}_{\alpha \in \Delta})$ of automorphisms of $N$ which stabilise $B, T$ and the set of $u_\alpha$. Furthermore,
$$ 1 \to \inn(N) \to \au (N) \to \aut(\psi_0(N)) \to 1 $$
is a split short exact sequence.
\end{prop}

The following result is also known.

\begin{prop} \cite[Equation 3.1e]{GV} \label{ouN_describe}
$\ou (N)$ is a finite index subgroup in $\GL \times S$ where $S$ is a finite group.
\end{prop}

To conclude the section, we describe the centre of $N$. \cite[II.1.5]{J} Let $(X,\Phi)$ denote the rational characters and roots of $N$, respectively. Then we define the centre of $N$ as 
\[Z(N) := \bigcap_{\alpha \in \Phi} \ker(\alpha). \]
We may represent this as follows: let $\Lambda$ denote the abelian group $X/\mathbb{Z}[\Phi]$ \textit{written multiplicatively}. We give $k[\Lambda]$ the structure of a commutative and cocommutative Hopf algebra by equipping it with comultiplication $\Delta(g) = g \otimes g$, counit $\epsilon(g) = 1$ and antipode $S(g) = g^{-1}$, for all $g \in \Lambda$. Then $Z(N) = \text{Sp}_k(k[\Lambda])$, which for any (commutative) $k-$algebra $A$ is equivalent to $\Hom_{\text{Grp}}(\Lambda, A^\times)$, with multiplication given by $(\phi \cdot \psi)(g) = \phi(g)\psi(g)$ for all $\phi, \psi \in \Hom_{\text{Grp}}(\Lambda, A^\times)$.

\subsection{Classification of disconnected reductive groups}
Now we have all the ingredients to classify disconnected reductive groups. A reductive group $G$ has finitely many components so that the group $\pi_0 (G)$ is finite. An isomorphism $G\rightarrow G'$ of reductive groups restricts to isomorphisms
\[
G_0\xrightarrow{\cong} G'_0, \quad \pi_0(G)\xrightarrow{\cong} \pi_0(G')
\]
of identity components and component groups. Hence, the task is to understand the sets $\sE(N, H)$ for all finite groups $H$ and connected reductive groups $N$. Let us start with possible couplings.
\begin{defn}
We say that a group $K$ has finitely many homomorphisms from finite groups (property $\mF$ for short) if for any finite group $H$ the set $\hom (H, K)$ has finitely many $K$-orbits under the conjugation by $K$.
\end{defn}

It is easy to come up with both $\mF$ and non-$\mF$ finitely presented groups, so we find the next question interesting.

\begin{prob}
Is the property $\mF$ decidable on the class of finitely presented groups?  
\end{prob}

We give a criterion, sufficient for our aims. We write $\mF$ for the class of groups with property $\mF$.

\begin{prop} The following statements hold for the class $\mF$.
\begin{enumerate}
    \item ${\mathfrak F}$ is closed under finite index subgroups.
        \item ${\mathfrak F}$ is closed under finite direct products.
    \item ${\mathfrak F}$ contains $\GL$ for any $n$.
    \item ${\mathfrak F}$ contains $\au (N)$ for any connected reductive group.
\end{enumerate}    
\end{prop}

\begin{proof}
The first two statements are obvious. The third statement is the Jordan-Zassenhaus Theorem. With all this knowledge, the last statement follows from Proposition~\ref{ouN_describe}.
\end{proof}

If $N$ is semisimple, $Z=Z(N)$ is finite. We can combine this proposition with Theorem~\ref{thm2}, replacing $\au(N)$ with $\ou(N)$ for future convenience.

\begin{cor} \label{groups_g0ss}
Suppose $N$ is a connected semisimple group and $H$ is a finite group.
The set $\sG(N,H)$ of isomorphism classes of semisimple groups $G$ with $G_0\cong N$ and $\pi_0(G)\cong H$ is finite and can be described as
\[
\sG(N,H) = \coprod_{\varpi \in \hom(H,\ou(N))/(\ou (N)\times \aut (H))} \dfrac{H^2 (H,  {}_{\varpi}Z)}{(\ou (N) \times \aut (H))_{\varpi}} .
\]
\end{cor}

Note that Corollary~\ref{groups_g0ss} works equally well for reductive $N$, except that it is not obvious why $\sG(N, H)$ must be finite. Another phenomenon is that this set is almost independent of the field $\bK$: for small $p$, the centre $Z$ can ``collapse" -- remember that it is the centre of the abstract group $N$ that we consider. Barring this collapse, the set is independent.

To contemplate the latter phenomenon, let $\mN$ be a reductive group scheme over $\bZ$ such that $N=\mN (\bK)$. We write
\[
\sG_\bK(\mN,H) \coloneqq \sG(\mN(\bK),H) = \sG(N,H).
\]
Let $k$ be the order of the finite group $H$. As before, $p$ is the characteristic of $\bK$. Let us write $k= \bar{k} p^{\nu_p(k)}$ so that $\bar{k}$ is the $p^\prime$-part of $k$. In the case of $p=0$, we agree that $\bar{k} = k$ and $0^{\nu_0(k)}=1$. We have a short exact sequence of abelian groups
\begin{equation} \label{map_gamma}
1 \rightarrow t_k(Z) \rightarrow Z \xrightarrow{\gamma_k \, : \, x \mapsto x^k} \im (\gamma_k) \rightarrow 1 \, ,
\end{equation}
where both the kernel and the image of $\gamma_k$ are in terms of groups (and not group schemes).
We are ready for the main theorem. 

\begin{thm} \label{gr_finite_thm} 
Suppose $N= \mN (\bK)$ is a connected reductive group, $c$ is the determinant of the Cartan matrix if its semisimple part $[\mN,\mN]$ and $H$ is a finite group of order $k$. The following statements hold.
  \begin{enumerate}
  \item The set $\sG_{\bK}(\mN, H)$ is finite and depends only on the characteristic of the field $\bK$. This justifies the notation $\sG_{p}(\mN,H)\coloneqq\sG_{\bK}(\mN,H)$.
  \item If $p$ does not divide $kc$, the set $\sG_p(\mN,H)$ does not depend on $p$.
  \item If $p$ divides $kc$, the set $\sG_p(\mN,K)$ is not bigger than $\sG_0(\mN,H)$.
   \end{enumerate}
An explicit formula for the $\sG_p(\mN, K)$ appears in the proof as \eqref{grp_formula}.
\end{thm} 

\begin{proof}
We build on Corollary~\ref{groups_g0ss}. The following ingredients do not depend on the field: the set $\hom(H,\ou(N))$ of all couplings, the group $\ou (N)\times \aut (H)$ and the action of this group on this set. 
 
However, the centre ${}_{\varpi}Z$ depends on $\bK$. The short exact sequence~\eqref{map_gamma} yields a long exact sequence in cohomology. By the standard application of the transfer map \cite[Corollary 10.2]{B}, $H^m(\gamma_{k}) = 0$ for all $m\geq 1$. It follows that
\[
\forall m \geq 2 \quad H^m (H,  {}_{\varpi}t_k (Z) ) \cong H^m (H,  {}_{\varpi}Z) \, .
\]
This is an isomorphism of $(\ou (N) \times \aut (H))_{\varpi}$-modules since the map $\gamma_k$ is an homomorphism of $(\ou (N) \times \aut (H))_{\varpi}$-modules. The $k$-torsion $t_k (Z)$ is finite and does not depend on the field $\bK$, only on its characteristic $p$. Hence, 
\begin{equation} \label{grp_formula}
\sG_p(\mN,H) = \coprod_{\varpi \in \hom(H,\ou(N))/(\ou (N)\times \aut (H))} \dfrac{H^2 (H, {}_{\varpi} t_k (Z))}{(\ou (N) \times \aut (H))_{\varpi}}    
 \end{equation}
is a finite set, independent of $\bK$.

Let us write $A$ for the centre of $[\mN,\mN]$. It is a finite abelian group whose order divides $c$, which can also be described as $A = Z [\mN (\bC),\mN(\bC)]$. The centre of $[N,N]$ is its $p'$-part $A_{p'}$. Now we need its subgroup
\[
B = Z [\mN (\bC),\mN(\bC)] \cap Z(\mN (\bC))_0
\]
by which it is ``attached'' to the connected centre. It gets mapped to $Z(\mN (\bK))_0 = (\bK^{\ast})^m$ for any $\bK$, yielding the following description of the centre
\[
Z(p) = A_{p'} \underset{B_{p'}}{\times} (\bK^{\ast})^m 
\]
The last two statements follow from the direct sum decomposition of $t_k(Z(p)$ and its cohomology:
\[
H^2 (H, t_k (Z(0))) =
H^2 (H, t_k (Z(p)) \oplus C) =
H^2 (H, t_k (Z(p))) \oplus H^2 (H, C)
\]
where $C=0$, if $p$ does not divide $kc$.
\end{proof}

\begin{ex}[Torus]
Suppose $\mN$ is abelian, so it is an $n$ dimensional torus. Then $N=(\bK^{\ast})^n$ and $\ou (N) = \au (N) \cong \GL$. Formula~\ref{grp_formula} becomes
\[
\sG_p(\mN,H) = \coprod_{\varpi \in \hom(H,\GL/(\GL\times \aut (H))} \dfrac{H^2 (H,  {}_{\varpi}C_{\overline{k}}^{n} )}{(\GL \times \aut (H))_{\varpi}} \, .
\]
The calculation of the cohomology is easy, but finding all couplings is equivalent to finding all integral representations of $H$. This could be very hard. For instance, let $H=C_n$ be the cyclic group. Consider the standard $\bQ H$-module $\bQ (e^{2\pi i/n})$. The number of distinct $\bZ H$-forms of this module is the class number of the cyclotomic field $\bQ (e^{2\pi i/n})$ (cf. \url{https://oeis.org/A061653}).
\end{ex}

\begin{ex}[Simple group] \label{ex sim} 
Consider a group $G$ with $G_0 = \text{Spin}_8$ and $H=S_3$. Assume that $\text{char}(p) \neq 2$, then the centre of $\text{Spin}_8$ is $C_2^2$. There are three possible classes of couplings, $\varpi$, in $\hom(S_3, S_3) / (\ou(N) \times \aut(S_3))$, one for every normal subgroup of $S_3$. The $6$-torsion of $Z(\text{Spin}_8) = C_2^2$ is $A := t_6(C_2^2) = C_2^2$.

For $\im(\varpi)=1$ we have $H^2(S_3, {}_\varpi A) = C_2^2$. Taking orbits under the action of ${(S_3 \times \aut (S_3))_{\varpi}}$ yields a trivial and non-trivial class. From this we obtain the direct product $\text{Spin}_8 \times S_3$ and a group arising from a non-split extension. For $\im(\varpi)=C_2$ we have $H^2(S_3, {}_\varpi A) = 1$ so we only get a semidirect product. For $\im(\varpi)=S_3$ we have $H^2(S_3, {}_\varpi A) = 1$ so we only get a semidirect product.

For this example, one can compute $\sG(A, S_3)$ and observe a one-to-one correspondence between it and $\sG(G_0, S_3)$ which respects both the image of the coupling and the splitness of the corresponding short exact sequence of groups. This is because $\ou(G_0) \cong \ou(A) \cong S_3$ and so the number of orbits is the same.

Finally, if $\text{char}(p) = 2$ then $Z(\text{Spin}_8) = 1$, all cohomology groups are trivial, and we only obtain one group for each coupling.
\end{ex}

\begin{ex}[Semisimple group] \label{ex ss}
Consider a group $G$ with $G_0 = SL_2^3$ and $H=S_4$. Assume that $\text{char}(p) \neq 2$, then the centre of $SL_2^3$ is $C_2^3$. As $\ou(SL_2^3) = S_3$ we have that the classes in $\hom(S_4, \ou(SL_2^3)) / (\ou(SL_2^3) \times \aut(S_4)) = \hom(S_4,S_3)/(S_3 \times S_4)$ are in correspondence with normal subgroups of $S_4$ containing $K_4 = C_2^2$. The $24$-torsion of $Z(SL_2^3) = C_2^3$ is $A := t_{24}(C_2^3) = C_2^3$. 

For $\im(\varpi)=1$ we have $H^2(S_4, {}_\varpi A) = C_2^6$. Furthermore, $(S_3 \times S_4)_{\varpi} = S_3 \times S_4$ and an application of the Burnside lemma gives that there are $20$ orbits of $C_2^6$ when acted on by this group. Thus, $\varpi$ with $\im(\varpi)=1$ corresponds to $20$ algebraic groups.

For $\im(\varpi)=C_2$ we have $H^2(S_4, {}_\varpi A) = C_2^3$. One calculates that $(S_3 \times S_4)_\varpi = C_2 \times S_4$ and the Burnside lemma gives $6$ orbits. Thus, $\varpi$ with $\im(\varpi)=C_2$ corresponds to $6$ algebraic groups.

For $\im(\varpi)=S_3$ we have $H^2(S_4, {}_\varpi A) = C_2^3$. One checks that $(S_3 \times S_4)_\varpi = \{(\sigma^{-1}, \tau) : \sigma \in S_3, \tau \in S_4, \varpi(\tau) = \sigma\}$ and the Burnside Lemma gives $8$ orbits. Thus, $\varpi$ with $\im(\varpi)=S_3$ corresponds to $8$ algebraic groups.

Note that for this example, we do not obtain a bijection between $\sG(G_0, S_4)$ and $\sG(A, S_4)$ with $\im(\varphi)$ being $1, C_2$ or $S_3$ as in the former we have $34$ groups and in the latter just $19$ groups. This is because $\ou(G_0) = S_3$ is not isomorphic to $\ou(A) = GL_3(\mathbb{F}_2)$, and the latter gives strictly less orbits when acting on $H^2(S_4, {}_\varpi A) = C_2^6$ (with $\im(\varpi)=1$).

Finally, if $\text{char}(p) = 2$ then $Z(SL_2^3) = 1$, all cohomology groups are trivial, and we only obtain one group for each coupling.
\end{ex}

\subsection{Subgroups in reductive groups}
To keep our notation consistent, in this section, we study a linear algebraic group $G$ over $\bK$ with identity component $G_0=N$ and the component group $\pi_0(G)=H$. Let $T_0\leq B_0$ be a maximal torus inside a Borel subgroup of $G_0$. The three key subgroups of $G$ are normalisers:
\begin{equation}T\coloneqq N_G(T_0,B_0), \ S\coloneqq N_G(T_0), \ B\coloneqq N_G(B_0). \end{equation}
Note that $T$ and $B$ are known as {\em a torus} and {\em a Borel subgroup} of $G$ \cite{DM}. By $W$ we denote the Weyl group of $G_0$. Let us summarise their properties.

\begin{lemma} \label{lemm_subs}
The following statements hold.
\begin{enumerate}
    \item $S \supseteq T\subseteq B$.
    \item $B_0$ is the identity component of $B$ and $B\cap G_0 = B_0$.
    \item $T_0$ is the identity component of $T$ and $T\cap G_0 = T_0$.
    \item $T_0$ is the identity component of $S$ and $S\cap G_0 = N_{G_0}(T_0)$.
    \item $T$ (hence, $S$ and $B$) meets every component of $G$.
    \item There is a short exact sequence $1\rightarrow W \rightarrow \pi_0 (S) \rightarrow H \rightarrow 1$.
    \item The set ${}^GT\coloneqq \underset{g\in G}{\cup}  gTg^{-1}$ is Zariski dense in $G$. 
\end{enumerate}
\end{lemma}

\begin{proof}
The first statement is obvious. The second follows from $N_{G_0}(B_0)=B_0$.

To prove~(3), consider $x\in T\cap G_0$. By~(2), $x\in B_0$. Then $x\in T\cap B_0 = N_{B_0}(T_0)=T_0$.

The second statement of~(4), $S\cap G_0 = N_{G_0}(T_0)$ follows from the definition of $S$. Then first statement is immediate: $S_0 = (S\cap G_0)_0= N_{G_0}(T_0)_0 =T_0$.

To prove~(5), pick $x\in G$. Since $xB_0 x^{-1}$ is a Borel subgroup in $G_0$, there exists $y\in B_0$ such that $y^{-1}B_0 y = xB_0 x^{-1}$. It follows that $yx \in B \cap G_0 x$. Now $yxT_0(yx)^{-1}$ is a torus inside $B_0$. There exists $z\in B_0$ such that $z^{-1} T_0 z = yxT_0(yx)^{-1}$. We have found $\tilde{z}\coloneqq zyx$ in $T \cap G_0 x$. 

To prove~(6), take the exact sequence $1\rightarrow N_{G_0} (T_0) \rightarrow S \rightarrow S/N_{G_0} (T_0)\rightarrow 1$ and quotient it by $T_0 = S_0$. The quotient group in the sequence is $H=\pi_0 (G)$ by a combination of (4) and (5).

For the final statement, it suffices to prove that ${}^GT$ is dense in an arbitrary component $G_0 x$. Let us use $\tilde{z}\in T \cap G_0 x$ from the proof of~(5). The $G_0$-action on  $G_0 \tilde{z}$ is the same as twisted action on $G_0$:
\[ g(h \tilde{z}) g^{-1} = g h (\tilde{z} g \tilde{z}^{-1})^{-1}\tilde{z} \; . \]
The conjugation by $\tilde{z}$ is the action by a diagram automorphism on $G_0$. The union of $G_0$-conjugates of $T_0$ is known to be dense for a semisimple $G_0$ \cite[Lemma 4]{Sp2}. It follows that it is dense in the reductive case too.
\end{proof}

The following, not very intuitive fact, is useful in the final chapter.

\begin{lemma} \label{lemma normalisers}
The normalisers of $T$ and $T_0$ in $G$ agree, that is $N_G(T)=S$.
\end{lemma}

\begin{proof}
The direction $N_G(T)\subseteq S$ is obvious. 

All Borel subgroups containing $T_0$ form a free orbit under the Weyl group $W = W(G_0)$ of $G_0$. 
For each  $w \in W$, we choose a lifting $\dot{w} \in N_{G_0}(T_0)$ so that 
\begin{equation} \label{setW}
\{ {}^wB_0\coloneqq \dot{w}B_0\dot{w}^{-1} \ \mid \ w\in  W\}    
\end{equation} is the full set of Borel subgroups, containing $T_0$.

Obviously, for any $w \in W$, we have 
$N_G(T_0,{}^wB_0)={}^wT\coloneqq \dot{w}T\dot{w}^{-1}$.
Observe that, in fact, $T={}^{w}T$.
Indeed, let $t \in T$. The image of the orbit map
\[\gamma_t : G_0 \longrightarrow G , \ \ g \mapsto g^{-1}tg\]
is connected, thus, lies inside the connected component $tG_0 = \gamma_t(1)G_0$.
In particular, $\gamma_t (\dot{w})\in tG_0$.
By Lemma~\ref{lemm_subs}, we have a bijection 
between connected components of $G$ and ${}^{w}T$ (note that $({}^{w}T)_0 = {}^{w}(T_0) = T_0$): 
\[
xT_0 \leftrightarrow yG_0
\quad \Longleftrightarrow \quad
xT_0 \subseteq yG_0
\quad \Longleftrightarrow \quad
xT_0 = {}^{w}T \cap  yG_0 \, .
\]
 Therefore, 
$\gamma_t (\dot{w})\in {}^{w}T \cap tG_0 = tT_0$. 
Write $\gamma_t (\dot{w}) = t\cdot t_0$ with $t_0\in T_0$
to observe that
\[
t({}^{w}B_0)t^{-1} = 
\dot{w}\gamma_t(\dot{w})B_0\gamma_t(\dot{w})^{-1} \dot{w}^{-1} = \dot{w} t t_0 B_0 t_0^{-1} t^{-1} \dot{w}^{-1} = \dot{w}B_0\dot{w}^{-1} = {}^{w}B_0\, .
\]
To prove $N_G(T)\supseteq S$,  pick $s\in S$. Conjugation by $s$ fixes $T_0$ hence moves around the set in~\eqref{setW}, hence fixes $T$ as well.
\end{proof}

Besides this, the Borel-Serre Theorem tells us that all linear algebraic groups (not necessarily reductive) admit ``large" finite subgroups:

\begin{lemma} \cite[Lemma 5.11]{BS}
Let $L$ be a linear algebraic group over $\bK$. Then there exists a finite subgroup $F$ of $L$ which intersects each connected component of $L$.
\end{lemma}

Making use of this theorem, Digne-Michel Theorem yields a useful finite subgroup in the reductive case:

\begin{thm} \cite[Theorem 1.17]{DM} \label{DM_th}
Every torus ${T} \subset {G}$ contains a finite subgroup ${A}$ such that $G = G_0 \cdot {A}$ and $A_0 := {A} \cap G_0$ is central in $G_0$. 
\end{thm}

The proof is constructive, so we give it here. We need some notation. 
Given a root of $G_0$ (that is a map $\alpha: T_0 \longrightarrow G_m$) define $X_\alpha: G_a \longrightarrow G_0$ to be the one-parameter subgroup associated to the root $\alpha$. Conjugation induces an action of $T$ on the roots. We denote this action by $^\sigma\alpha$:
\[
(^\sigma\alpha)(t) \coloneqq \alpha(^{\sigma^{-1}}t) = \alpha(\sigma^{-1}t\sigma) \ \mbox{ for all } \ \sigma \in T \ \mbox{ and  roots } \ \alpha \, . 
\]

\begin{proof}
Let $\widetilde{A} \coloneqq \{\sigma \in {T} \,\mid\,  {}^{\sigma}X_\alpha(\lambda) = X_{^\sigma\alpha}(\lambda)$ for all simple $\alpha \}$. 

The subgroup $\widetilde{A}$ meets every connected component of ${T}$. Indeed, pick $t \in {T}$. Then $^tX_{\alpha}(\lambda)$ is a multiple of $X_{^t\alpha}(\lambda)$, say $^tX_{\alpha}(\lambda) = n \cdot X_{^t\alpha}(\lambda)$ for some $n \in \bK^\times$. Choose $t_0 \in T_0$ such that $^{t_0}\big(\hspace{0ex}^tX_{\alpha}(\lambda)\big) = n^{-1} \cdot \hspace{0ex}^tX_{\alpha}(\lambda)$. Then $t_0t \in \widetilde{A}$ with $t_0t$ and $t$ living in the same connected component of $T$. Hence, ${T} = T_0 \cdot \widetilde{A}$ and, by Lemma~\ref{lemm_subs}, $G = G_0 \cdot \widetilde{A}$.

On the other hand, for all $t \in T_0$ we have $X_{^{t}\alpha}(\lambda) = X_{\alpha}(\lambda)$. Hence, $\widetilde{A} \cap G$ consists of those elements of $T_0$ which stabilise all one-parameter subgroups, i.e., stabilise all of $G_0$. Thus, $\widetilde{A} \cap G_0 \in Z(G_0)$. Finally, by Borel-Serre Theorem, there exists a finite subgroup ${A} \subset \widetilde{A}$ such that $\widetilde{A} = \widetilde{A}_0 \cdot {A}$. The finite subgroup ${A} \subset \widetilde{A} \subset {T}$ satisfies the required conditions.
\end{proof}

We finish this section with further immediate properties of the subgroup $A$ that we find useful.

\begin{cor}
The finite subgroups $A$ and $A_0$, constructed in  Theorem~\ref{DM_th}, satisfy the following properties.
\begin{enumerate}
    \item The subgroup $A_0$ is normal in $G$.
    \item The quotient group $G/A_0$ is a semidirect product $(N/A_0)\rtimes H$.
    \item The group $G$ can be represented as a crossed product $N\#_{\phi,f} H$ where the function $f$ takes values in $A_0$.
\end{enumerate}
\end{cor}

It would be interesting to connect the finite group to our analysis of extensions, in particular, to formula~\eqref{grp_formula}. One obvious observation is that one can find $A$ such that $A_0\subseteq t_k (Z(G_0))$. To facilitate further research, we ask the following question.
\begin{prob}
Given a group $G$, classify all minimal subgroups $A$ that satisfy Theorem~\ref{DM_th}. Are all of them conjugates? 
\end{prob}

\section{Representation theory}
In this section, we keep our previous notations.
In Section~\ref{ch3s4} we assume that the algebraically closed field $\bK$ has characteristic zero. The results of the first three subsections do not require any restrictions on the characteristic.

\subsection{Mackey normal subgroup machine} \label{ch3s1}

This machine is a variation of Clifford Theory that yields the classification of irreducible representations of $G$ from the irreducibles of $G_0$ and extra data \cite{Mackey}. We recall it following the recent exposition by Achar, Hardesty and Riche \cite{AHR}, where an interested reader can find proofs of the lemmas and the theorem. 

Let $\mathbf{T}_0$ be the abstract torus of $G_0$. The isomorphisms between two Borels $B_0 \xrightarrow{\cong} B'_0={}^{g}B_0$, $g\in G_0$ form a coset $gB_0$. All of them descend to the same isomorphism between abelianisations $B_0/[B_0, B_0] \xrightarrow{\cong} B'_0/[B'_0, B'_0]$. Hence, we can identify all of them and define the abstract torus as $\mathbf{T}_0 \coloneqq B_0/[B_0, B_0]$. It has an obvious action of $G$ with $G_0$ acting trivially.

Let $\mathbf{X}\coloneqq \hom (\mathbf{T}_0,G_m)$ be the weight lattice of $G_0$ and $\mathbf{X}^+$ the set of dominant weights. The group $G$ acts on $\mathbf{X}$ on the left via 
\[
({}^g \lambda)(t) = \lambda(g^{-1}tg) \ \mbox{ for all } \ g \in G, \lambda \in \mathbf{X} \, .
\]
Further, $G$ acts on $G_0$-modules. Given a $G_0$-module $L=(L,\rho)$, we define
\[
{}^gL \coloneqq (L, x\mapsto \rho (gxg^{-1})) \ \mbox{ for all } \ g \in G \, .
\]

\begin{lemma}
  The natural bijection $\mathbf{X}^+ \to \Irr(G_0)$, sending $\lambda$ to the isomorphism class of the irreducible highest weight module $L(\lambda)$ is $G$-equivariant, i.e, ${}^gL(\lambda) \cong L({}^g\lambda)$. Furthermore, $G_0$ acts trivially on both sets, so that the component group $H=\pi_0(G)$ acts on them.  
\end{lemma}

For a dominant weight $\lambda \in \mathbf{X}^+$, by $H^\lambda, \, G^\lambda$ we denote its stabilisers. Note that 
\[1 \to G_0 \to G^\lambda \to H^\lambda \to 1\] 
is a short exact sequence. Let us choose a transversal in $A$:
\[
\iota: H \rightarrow A \ \mbox{ is a section of } \ A \twoheadrightarrow H 
\ \mbox{ s. t. } \ \iota(1_{H}) = 1_G \, .
\]
With this transversal, for all $x,y\in H^\lambda$ we have
\[
\iota(x)\iota(y)=\iota(xy) \gamma (x,y) \ \mbox{ for some } \ \gamma(x,y)\in A_0 \subseteq Z(G_0)\, .
\]
It follows that for the simple $G_0$-module $(L (\lambda),\rho)$
\[
{}^{\gamma(x,y)}L(\lambda)=L(\lambda) \ \mbox{ and } \ \rho(\gamma(x,y)): L(\lambda) \xrightarrow{\cong} {}^{\gamma(x,y)}L(\lambda)
\]
is a multiplication by a scalar. Twisting both sides by $\iota (xy)$, we get another scalar
\[
\phi_{x,y} \in \bK^*, \ \ 
\phi_{x,y} : {}^{\iota(x)\iota(y)}L(\lambda)\xrightarrow{\cong} {}^{\iota(xy)}L(\lambda) = {}^{\iota(x)\iota(y)}L(\lambda) \, .
\]
Now for each $x \in H^\lambda$ we choose $\theta_x : L(\lambda) \xrightarrow{\cong} \,^{\iota(x)}L(\lambda)$, an isomorphism of simple $G_0$-modules. 
Observe that 
\[ 
\theta_{xy}, \; \phi_{x,y} \circ \,^{\iota(x)}\theta_y \circ \theta_x :
L(\lambda) \xrightarrow{\cong} {}^{\iota(xy)}L(\lambda)
\]
are both isomorphisms of $G_0$-modules. Therefore, they differ by a scalar. Denote this scalar by $\alpha (x,y)\in \bK^\ast$.
\begin{lemma}
The following function is a 2-cocycle:
\[
\alpha: H^\lambda \times H^\lambda \rightarrow \bK^\ast, \ \ (x,y)\mapsto \alpha (x,y) \, .
\]
\end{lemma}
Let $\sA^\lambda$ be the twisted group algebra $\bK\ast_\alpha H^\lambda$. Its modules are projective representations of $H^\lambda$ with the cocycle $\alpha$. It has a basis
\[
\{ \rho_x \mid x\in H^\lambda \} \ \mbox{ with } \ \rho_x \rho_y = \alpha (x,y) \rho_{xy} \, .
\]

\begin{thm} \cite{AHR} \label{AHR_theorem}
The following statements hold.
\begin{enumerate}
    \item If $E\in\Irr (\sA^\lambda)$, then  $E \otimes L(\lambda)$ is a simple $G^\lambda$-module and 
    \[L(\lambda, E) \coloneqq \textup{Ind}_{G^\lambda}^G(E \otimes L(\lambda))\]
    is a simple $G$-module.
    \item The following function is a bijection:
    \[
    \coprod_{\lambda \in \mathbf{X}^+} \Irr (\sA^\lambda) \rightarrow \Irr (G), \quad
    (\lambda, E) \mapsto L(\lambda, E) \, .
    \]
    \item 
The restriction from $G$ to $G_0$ is given by the formula
\[\textup{Res}^G_{G_0} (L(\lambda,E)) \, \cong \, \dim(E)\Big( \bigoplus^r_{i=1} \ ^{g_i} L(\lambda)\Big) \]
where $g_1,\ldots, g_r$ are representatives in $G$ of the cosets in $G/G^\lambda$.    
\end{enumerate}
\end{thm}

Note that the direct sum of the conjugates of $L(\lambda)$ in the restriction formula can be written as $\oplus^r_{i=1}  L(\lambda_i)$ where $\lambda_i$ are all images of $\lambda$ under the action of various graph automorphisms of the Dynkin diagram.

Let us compare this data for the group $G$ with similar data for its torus $T$: we will keep the same letter, just adding a superscript ${}^{\diamond}$ for the torus data. First, $\pi_0 (T) =\pi_0 (G) = H$. Hence, the chosen transversal $\iota$ serves both $G$ and $T$ equally so that $\gamma^{\diamond} = \gamma$. Simple $T_0$-modules $L^{\diamond}(\lambda)$ are one dimensional and parameterised by all the weights $\lambda \in \mathbf{X}$. For a weight $\lambda$, let $\lambda^{\flat} \in \mathbf{X}^+$ be the unique dominant weight in the $W$-orbit of $\lambda$ (here, $W$ is the Weyl group of $G_0$). We identify $L^{\diamond}(\lambda)$ with the extremal weight subspace $L(\lambda^\flat)_{\lambda}$ of $L(\lambda)$. 
\begin{prop} \label{AHR_prop}
The following statements comparing the data for the $G$-module $L(\lambda^\flat)$ and the $T$-module $L^{\diamond}(\lambda)$ hold.
\begin{enumerate}
    \item $H^{\lambda^\flat}=H^{\lambda}$ and $T^{\lambda^\flat}=T^{\lambda}$.
    \item $\phi^{\diamond}_{x,y}=\phi_{x,y}$ for all $x,y\in H^\lambda$.
    \item $\theta^{\diamond}_{x}=\theta_{x}|_{L(\lambda^{\flat})_\lambda}$ for all $x \in H^\lambda$.
    \item $\alpha^{\diamond} ({x,y})=\alpha ({x,y})$ for all $x,y\in H^\lambda$.
    \item The following function is a bijection:
    \[
    \coprod_{\lambda \in \mathbf{X}} \Irr (\sA^{\lambda^\flat}) \rightarrow \Irr (T), \quad
    (\lambda, E) \mapsto L^{\diamond}(\lambda, E) \coloneqq \ind_{T^\lambda}^T (E\otimes L^{\diamond}(\lambda)) \, .
    \]
    \item We have the following bijection, where the Weyl group acts on $\lambda$ in $L^{\diamond}(\lambda,E)$:
    \[
    \Irr (T)/W \rightarrow \Irr (G), \quad  W \cdot L^{\diamond}(\lambda,E) \mapsto L(\lambda^\flat, E) \, .
    \]
    \end{enumerate}
\end{prop}
\begin{proof}
The first statement follows from the fact that a simple module $L(\lambda)$ is uniquely determined by any of its extremal weights. 

The next three statements follow from our identification of 
$L^{\diamond}(\lambda)$ with the extremal weight subspace $L(\lambda^\flat)_{\lambda}$ of $L(\lambda)$. 

The last two statements follow from the first four statements as well as an application of Theorem~\ref{AHR_theorem} to both $T$ and $G$.
\end{proof}
We finish this section with an interesting problem. 
\begin{prob}
Find the axioms that the function
\[
[\alpha (\lambda)] : \mathbf{X}^+
\rightarrow
\coprod_{S\leq H} H^2 (S, \bK^\ast), \quad \lambda \mapsto [\alpha] \in H^2(H^\lambda,\bK^\ast)
\]
satisfies. How does one compute it algorithmically? Which of the functions correspond to groups $G$?
\end{prob}

\subsection{Equivariant sheaves on finite \texorpdfstring{$G$}{G}-sets} \label{ch3s2}
The category of equivariant sheaves on finite sets becomes useful here, cf. \cite{JMR, With} in the context of finite groups.
Let us discuss it here in the context of disconnected reductive groups. 

\begin{defn}
The objects of {\em the global representation category} $G\GRe$ are pairs $(X,V)$ where $X$ is a finite $G$-set and $V$ is a finite-dimensional $G$-equivariant vector bundle on $X$. The morphisms from $(X,V)$ to $(Y,U)$ are pairs $(f,\phi)$ where $f:X\rightarrow Y$ is a $G$-equivariant map and $\phi: V \to f^{\ast}(U)$ is a morphism of $G$-equivariant sheaves on $X$.
\end{defn}

Notice that $G_0$ has to act trivially, so the $G$-sets are just $\pi_0 (G)$-sets. Observe that $\phi$ consists of $G$-equivariant linear maps $\phi_x : V_x \rightarrow U_{f(x)}$. This description gives a clear view of the composition of morphisms
\[
(f',\phi'_y) \circ (f,\phi_x) = (f' \circ f,  \phi'_{f(x)} \circ \phi_x) \, . 
\]

The category has a coproduct and a symmetric tensor product, given by
\[
(X,V) \oplus (Y,U) \coloneqq (X\coprod Y, V \coprod U), \ \
(X,V) \otimes  (Y,U) \coloneqq (X \times Y, V \boxtimes U)
\]
where $(V \boxtimes U)_{(x,y)} =  V_x \otimes_{\bK} U$. 

This category is interesting even for a connected group $G$. Its objects are collections of $G$-modules $(V_x)_{x\in X}$, parameterised by finite sets. We leave it to an interested reader to work out morphisms, the coproduct and the tensor product.

\begin{prop}
  The global sections functor
  \[
  \Gamma : G\GRe \rightarrow G\Mod \, , \quad (X,V) \mapsto \Gamma(X,V) \coloneqq \bigoplus_{x\in X} V_x
  \]
  is symmetric monoidal.
It admits a right adjoint functor $\Gamma_{\ast}$.
\end{prop}
\begin{proof}
  Write the $G$-set $X$ as a union of orbits, choosing a representative $x(i)$ for each orbit. Each $V_{x(i)}$ is a module over the stabiliser
  $G_{x(i)}$. The global section is ``naturally'' a direct sum of the induced representations
  \[
  \Gamma (X,V) \quad = \quad \oplus_{i} \; \ind_{G_{x(i)}}^G  V_{x(i)} \, .
  \]
  The functoriality is clear. The symmetric monoidal structure easily comes from the standard twist $v \otimes u \mapsto u \otimes v$.

  The right adjoint is the functor  giving the sheaf on a point
  \[
  \Gamma_{\ast} (M) \coloneqq (\mbox{point}, M) \, .
  \]   
\end{proof}

There is no left adjoint to $\Gamma$ but there is a functor in the opposite direction, extending the assignment
\[
\gamma: L(\lambda, E) \ \mapsto \ ( G/G_{\lambda} , \widetilde{L(\lambda) \otimes E}), \quad  
\widetilde{L(\lambda) \otimes E}
\coloneqq
G \underset{G_{\lambda}}{\times}  (L(\lambda) \otimes E) \, , 
\]
defined on simple modules. The functor is defined only on the full subcategory of semisimple modules
$\gamma: G\Mod_{ss} \rightarrow G\GRe$ by 
\[
V \mapsto \bigoplus_{L\in \Irr (G)} L \otimes \hom_G (L,V) \mapsto
\coprod_{L=L(\lambda, E)\in \Irr (G)}  ( G/G_{\lambda} , (\widetilde{L(\lambda) \otimes E}) \otimes \hom_G (L,V)) \, . 
\]

\begin{defn}\label{global_rep_ring_defn}
The Grothendieck group $K(G\GRe)$ is a ring: its addition comes from the coproduct, 
its multiplication comes from the tensor product.
{\em The global representation ring}  $\sR(G,G_0)$ is 
a ``gauged'' Grothendieck ring of $K(G\GRe)$, i.e., 
the quotient ring $K(G\GRe)/I$ where the ideal $I$ is generated by ``gauges''
for all $(X,V), (X,U) \in G\GRe$.
\begin{equation} \label{relation}
[(X,V)] + [(X,U)] - [(X, V\oplus U)]   \, .
\end{equation}
\end{defn}

Now we characterise the representation ring of $G$.

\begin{thm} \label{gl_rep_map}
  The global sections functor $\Gamma$ descends to a surjective ring homomorphism $\sR(G,G_0)\twoheadrightarrow \sR(G)$.
\end{thm}
\begin{proof}
Since $\Gamma$ is a symmetric monoidal functor, it descends to a ring homomorphism between Grothendieck rings 
\[
  K(\Gamma) : K(G\GRe) \rightarrow K(G\Mod) \, .  
  \]
Since  $K(\Gamma)$ vanishes on the relations~\eqref{relation}, it descends to a ring homomorphism $\sR(G,G_0)\twoheadrightarrow \sR(G)$. The functor $\gamma$ yields a right inverse to this map, since $\Gamma (\gamma (M)) \cong M$. Thus, this map is surjective.
\end{proof}  

\subsection{Classical description of the representation ring} \label{ch3s3}
In this short section, we prove the following theorem:
\begin{thm}\label{isom R(G) and R(T)^N}
The restriction to the torus $T$ defines a ring isomorphism \newline
$\varphi: \sR(G)\xrightarrow{\cong} \sR(T)^{N_G(T)}$.    
\end{thm}

\begin{proof}
The restriction functor $\res^G_T$ yields a well-defined ring homomorphism $\varphi: \sR(G) \rightarrow \sR(T)$. Note any $T$-module $\res^G_T (V)$ is $N_G(T)$-invariant, i.e., isomorphic to its own twists by any $x\in N_G(T)$. Thus, the image of $\varphi$ lies in $\sR(T)^{N_G(T)}$. 

To deduce that $\varphi$ is injective, 
recall 
that ${}^GT$ is dense in $G$ (see Lemma~\ref{lemm_subs}). If $[V]\in \ker \varphi$, then its character $\chi_V$ is identically zero on $T$. Instead of $T$, one can restrict to any conjugate ${}^gT$, $g\in G$, also getting zero. Thus, $\chi_V$ is identically zero on ${}^{G}T$ and, hence, on $G$. Hence, $[V]=0$.

To prove surjectivity of $\varphi$, we recall from Lemma \ref{lemma normalisers} that $N_G(T)=S$. By Proposition~\ref{AHR_prop}, the elements
\[
X(\lambda, E) \coloneqq \sum_i L^\diamond (\lambda_i , E), \quad \lambda\in  \mathbf{X}^+, E \in \Irr (\sA^\lambda)\, ,
\]
where $\lambda_i$ runs over the joint $\langle W,H \rangle$-orbit of $\lambda$ (the same as $\pi_0 (S)$-orbit), form a $\bZ$-basis of $\sR(T)^{N_G(T)}$. Now observe that 
\[
\varphi (L(\lambda, E)) =  X(\lambda, E) + \sum_{F\in \Irr (\sA^\mu), \ \mu \prec \lambda} N_{\mu, F} X(\mu, F)
\]
where $N_{\mu,F}\in\bZ$ and we use the partial order~\eqref{order_wt} on the dominant weights. The easy induction, using this order on weights, now proves the surjectivity of $\varphi$.
\end{proof}

Notice that if $G$ is connected, then the action of the normaliser $N_G(T)$ yields the action of the Weyl group $W=N_G(T)/T$ so that $\sR(T)^{N_G(T)}=\sR(T)^W$ and we have a well-known classical statement. In this light, it is tempting to call the image of the group homomorphism
\[ \psi: N_G (T) \rightarrow \out (T)
\]
the Weyl group of (disconnected) $G$ but we shall resist this temptation.
\begin{prob}
Clarify the structure of the image of $\psi$ and its relation to $\pi_0 (S)$, cf. Lemmas~\ref{lemm_subs} and \ref{lemma normalisers}.
\end{prob}

\subsection{Knutson Index} \label{ch3s4}
For this section, we add a restriction that the algebraically closed field $\bK$ has characteristic zero. Now we use the theory developed in the previous chapter to study the Knutson Index for reductive groups \cite{JMR, M}. 

\begin{defn} 
The Knutson Index of a $G$-module $M$ is the smallest positive integer $n$ such that there exists a virtual module $V$ satisfying 
$$ M \otimes V = n \cdot R(G) $$ 
where $R(G)$ is the regular module of $G$. We denote it by $\mathcal{K}(M)$. The Knutson Index of $G$ is the lowest common multiple of the Knutson Indices of all simple $G$-modules and is denoted by $\mathcal{K}(G)$. We say that $G$ is of Knutson type if $\mathcal{K}(G) = 1$. 
\end{defn}

Note that $ M \otimes R(G) = \dim (M) \cdot R(G) $, so that $\sK(M)$ divides $\dim (M)$. However, the dimensions of simple $G$-modules grow and it is not clear a priori that $\sK (G)$ is finite.

Let us first consider the case of a connected reductive group $N$. We can further assume that $N$ is semisimple: just replace $N$ by its semisimple part $N_{ss} = [N, N]$ and use the following result.

\begin{prop} \label{prop ss}
Let $N_{ss}$ be the semisimple part of $N$. Then $\mathcal{K}(N) = \mathcal{K}(N_{ss})$.
\end{prop}

\begin{proof}
We start with the central product decomposition 
\[ N \cong Z(N)_0 \bullet_A N_{ss} \coloneqq \dfrac{Z(N)_0 \times N_{ss}}{ \{(x,x^{-1})\mid x\in A\} }, \quad
A = Z(N)_0 \cap N_{ss}\] 
where $Z(N)_0$ is the connected component of the centre. The irreducible representations of $N$ have the form 
\[
M \otimes \bK_\alpha, \ \mbox{ where } \ (M,\rho) \in \Irr (N_ss), \ \alpha: Z(H)_0 \to \bK^\times, \ \rho\mid_A = \alpha\mid_A\]
Let $\varpi : \Irr (Z(N)_0) \rightarrow \Irr (A)$ be the natural restriction. The characters of $A$ yield the following decomposition of the regular representation 
\begin{equation} \label{decompo}
R(N) \cong \bigoplus_{\beta\in \Irr (A)} \Big( \bigoplus_{\alpha\in \varpi^{-1}(\beta)} R(N_{ss})_\beta \otimes \bK_\alpha \Big) 
\ \mbox{ or } \ 
R(N)_\alpha \cong R(N_{ss})_{\varpi(\alpha)} \otimes \bK_\alpha \, .
\end{equation}
It remains to observe that $\sK (M) = \sK (M\otimes \bK_\alpha)$. 
Indeed, let $\beta\in \Irr (A)$ be the character of $A$ afforded by $M$.  Then
\[
M \otimes V \cong k R(N_{ss}) \ \Longleftrightarrow \ \forall \gamma\in \Irr (A) \ M \otimes V_\gamma \cong k R(N_{ss})_{\beta\gamma}
\]
and the similar thing is true for $N$, just using $Z(N)_0$-homogeneous components:
\[
(M\otimes \bK_\alpha) \otimes V \cong k R(N) \ \Longleftrightarrow \ \forall \gamma\in \Irr (Z(N)_0) \ (M\otimes \bK_\alpha) \otimes V_\gamma \cong k R(N)_{\alpha\gamma} \, .
\] 
In light of~\eqref{decompo}, both require solving the same equations on virtual modules $V_\gamma$.
\end{proof}

Suppose $N$ is semisimple and connected. We have a partial order on the simple $N$-modules given by the following partial order on weights
\begin{equation} \label{order_wt}
 \lambda \succeq \mu \quad \text{if} \quad \lambda - \mu \quad \text{is a sum of positive roots}   
\end{equation}
Its significance is the tensor product decomposition \cite[Chapter 13.3]{FH}
\begin{equation} \label{tensor_dir_sum}
L(\lambda) \otimes L(\mu) = L(\lambda+\mu) \oplus \bigoplus_{\nu \prec \lambda+\mu} N_\nu L(\nu).   
\end{equation}
where $N_\nu$ denotes the multiplicity of $L(\nu)$.

Define $\lambda^T \coloneqq - w_0 (\lambda)$, where $w_0\in W$ is the longest element. Its significance is the duality:  $L(\lambda)^\ast \cong L(\lambda^T)$.
\begin{prop} \label{prop connected inverse}
Let $\lambda,\mu \in\mathbf{X}^+$ and $\nu=\mu  - \lambda^T$. Then 
\begin{equation} \label{tensor_dir_sum1}
L(\lambda) \otimes L(\mu) = 
\begin{cases}
  L(\nu) \oplus \bigoplus_{\xi \succ \nu} N_\xi L(\xi)   &  \ \mbox{ if } \ \nu\in \mathbf{X}^+\, , \\
  \bigoplus_{\xi \succ \nu} N_\xi L(\xi)   & \ \mbox{ if } \ \nu \not\in \mathbf{X}^+\, .
\end{cases}
\end{equation}
\end{prop}

Note that if both $\mu  - \lambda^T$ and $\lambda-\mu^T$ are dominant, then $(\mu  - \lambda^T)^T = \mu^T - \lambda$ is dominant too. This forces $\mu^T = \lambda$ and $\nu=0$.

\begin{proof}
The multiplicity $[L:M]$ of a module $L$ in a simple module $M$ is equal to
$\dim (\hom_N(L,M)) = \dim (\hom_N(M,L))$. Since 
\begin{align*}
\hom_N(L({\xi}), L({\lambda}) \otimes L({\mu})) &
\cong 
\hom_N(L({\xi}), \hom_{\bK}(L({\lambda})^* , L({\mu}))) \cong \\
\cong & \hom_N(L({\xi}) \otimes L(\lambda^T), L(\mu))\, ,
\end{align*} 
we can conclude that for all $\xi \in\mathbf{X}^+$
\[
[L(\xi) : L({\lambda}) \otimes L({\mu})] \ = \ [L(\mu) : L({\xi}) \otimes L(\lambda^T) ]\, .
\]
If $\nu \in \mathbf{X}^+$, the desired $[L(\nu) : L({\lambda}) \otimes L({\mu})]=1$ is guaranteed. 

Now make no assumption on $\nu$. Suppose $L(\xi)$ appears in $L(\lambda) \otimes L(\mu)$. Then $L(\mu)$ appears in $L(\xi) \otimes L(\lambda)^*$. By~\eqref{tensor_dir_sum},
$$ \mu \preceq \lambda^T + \xi \implies \nu \preceq \xi \, ,$$
as claimed.
%
%
\end{proof}


%

We now want to study the Knutson Index for the general case, where $G$ is a disconnected reductive group and $N$ is its identity component. By Proposition~\ref{prop ss} we can assume that the identity component $N$ of $G$ is semisimple. \\

Now we construct a total order on the irreducible representation of $N$ that refines the partial order $\succeq$. First, introduce the height function 
\[
\sL: \mathbf{X}^+ \rightarrow \bN\, , \quad \lambda \mapsto \langle \lambda, \rho \rangle\, , 
\]
where $\langle \cdot, \cdot \rangle$ is the dot product in the basis of simple roots and $\rho$ is the half-sum of positive roots.

\begin{defn}
Let $\lambda,\mu\in \mathbf{X}^+$. We set $\lambda > \mu$ if $\sL(\lambda) > \mathcal{L}(\mu)$ or $\mathcal{L}(\lambda) = \mathcal{L}(\mu)$ and $\lambda$ comes after $\mu$ in the lexicographic order in the basis of fundamental weights, i.e., $\lambda = \sum_i l_i \varpi_i$, $\mu=\sum_i m_i \varpi_i$, for some $j \quad l_j > m_j$ and for all $i<j \quad l_i = m_i$.
\end{defn}

Note that this total order refines our previous partial order. If $\lambda \succ \mu$ then $\lambda - \mu$ is a sum of positive roots. Therefore, $\mathcal{L}(\lambda - \mu) > 0$ and, as the height function is linear, we conclude that $\mathcal{L}(\lambda) > \mathcal{L}(\mu)$.

Also, note that the total order preserves addition. Hence, it can be extended to all of $\mathbf{X}$ in the obvious way.

\begin{lemma} \label{lemma orbit}
Consider a finite dimensional $N$-module of the form 
\[ M= L(\lambda_1) \oplus \bigoplus_{i>1} N_i L(\lambda_i)\, , \] 
where $\lambda_1^T > \lambda_k^T$ for all $k>1$. Then $\mathcal{K}(M) = 1$.
\end{lemma}

\begin{proof}
Let us write an $N$-module as $V = \oplus_{\mu} X_\mu L(\mu)$. 
We denote its $\nu$-truncation by $V_{\mu\leq\nu} \coloneqq \oplus_{\mu\leq\nu} X_\mu L(\mu)$. 
We will construct an $N$-module $V$ such that $M \otimes V \cong R(N)$. The idea of the construction is to fix the coefficient $X_\mu$ one at a time, making sure that $(M \otimes V)_{\leq \nu} \cong R(N)_{\leq \nu}$ for larger and larger $\nu$. 



Now we essentially do induction on $\nu$, changing $V$ every time. For $\nu=0$ we set $V=L(\lambda_1^T) \cong L(\lambda_1)^\ast$. Clearly, $(M \otimes V)_{\leq 0} \cong L(0) \cong R(N)_{\leq 0}$.

Suppose we have accumulated $V$ such that $(M \otimes V)_{< \nu} \cong R(N)_{ < \nu}$. Let $\mu \coloneqq \nu + \lambda_1^T$. We can change $X_\mu$ in such a way that $(M \otimes V)_{\leq \nu} \cong R(N)_{\leq \nu}$. With higher and higher $\nu$, we are using higher and higher $\mu$. Hence, the process will produce a well-defined $V$ with $M \otimes V \cong R(N)$ after countably many steps.

Why can we change $X_\mu$?
By Proposition~\ref{prop connected inverse}, 
\[
L({\lambda_1}) \otimes L(\mu) \ \cong \ L(\nu) \ \oplus \ \mbox{ ``bigger than } \nu \mbox{ terms''}
\]
so we need to examine $L({\lambda_k}) \otimes L(\mu)$ for $k>1$. Again by Proposition~\ref{prop connected inverse}, 
\[
L({\lambda_k}) \otimes L(\mu) \ \cong \ \bigoplus \ \mbox{ ``bigger or equal than } \mu-\lambda_k^T \mbox{ terms''}\, .
\]
Since $\lambda_1^T > \lambda_k^T$, we conclude that $\mu-\lambda_k^T< \mu-\lambda_1^T=\nu$ and $L({\lambda_k}) \otimes L(\mu)$ contributes only larger terms so that 
\[
M \otimes L(\mu) \ \cong \ L(\nu) \ \oplus \ \mbox{ ``bigger than } \nu \mbox{ terms''} \, .
\]
This allows us to change $X_\mu$ to achieve our objective.
%
\end{proof}

In particular, Lemma~\ref{lemma orbit} applies to simple $N$-modules:
\begin{cor}
Every connected reductive group $N$ is of Knutson type.
\end{cor}

We now consider the general case where $G$ is any reductive group and state the main theorem of this section.

\begin{thm} \label{Knutson_finite}
The Knutson Index of a reductive group $G$ is finite.
\end{thm}

\begin{proof}
Let $L(\lambda, E)$ be a simple $G$-module. By Theorem~\ref{AHR_theorem}, its restriction to $G_0$ is $\dim(E) \cdot M$ where $M=\oplus_i L(\lambda_i)$ where the summation goes over the orbit of $\lambda$ under the automorphisms of the Dynkin diagram. The module $M$ satisfies the condition of  Lemma~\ref{lemma orbit} so that 
$$ \mathcal{K}(\dim (E) \cdot M) = \dim(E) \cdot \mathcal{K}(M) = \dim(E) \, .$$
It follows that there exists a $G_0$-module $V$ such that $L(\lambda, E) \otimes V \cong \dim(E) \cdot R(G_0)$. Inducing this from $G_0$ to $G$ and using the tensor identity \cite[Proposition I.3.6]{J} gives 
$$ L(\lambda, E) \otimes \text{Ind}_{G_0}^G (V) = \dim(E) \cdot R(G) $$
It follows that $\mathcal{K}(L(\lambda, E))$ divides $\dim(E)$.
Since there are finitely many possible $E$, $\sK(G)$ is finite. 
\end{proof}

\begin{cor} \label{cor div}
The Knutson Index of $G$ divides $|H|$, where $H=\pi_0 (G)$.
\end{cor}

\begin{proof}
From the proof of Theorem~\ref{Knutson_finite}, it is clear that $\mathcal{K}(G)$ divides the lowest common multiple of the dimensions of all simple $\sA^\lambda$-modules for all $\lambda \in \mathbf{X}^+$. The degrees of simple $\sA^\lambda$-modules divide the order of $H^\lambda$.
Since all $H^\lambda$ are subgroups of $H$, we conclude that $\mathcal{K}(G)$ divides $|H|$.
\end{proof}

\begin{lemma}
If $1 \to N \to G \to H \to 1$ is a short exact sequence of groups then $\mathcal{K}(H)$ divides $\mathcal{K}(G)$.
\end{lemma}

\begin{proof}
Suppose $\mathcal{K}(G) = n$. Let $M$ be a simple $H$-module. Let $\textup{Inf}(M)$ denote the inflation from $H$-Mod to $G$-Mod. Then $\textup{Inf}(M)$ is  a simple $G$-module, thus there exists a $G$-module $V$ such that $\textup{Inf}(M) \otimes V = n R(G)$. Taking $N$-invariants and viewing the result in terms of $H$-modules gives $M \otimes V^N = n R(H)$. Therefore, $\mathcal{K}(H) \mid \mathcal{K}(G)$. 
\end{proof}

\begin{cor}
Let ${G}$ be a reductive group. Then $\mathcal{K}(\pi_0(G))$ divides $\mathcal{K}({G})$.
\end{cor}

\begin{ex}
Recall from Example \ref{ex sim} the groups with $G_0 = \text{Spin}_8$ and $H = S_3$. By the proof of Corollary \ref{cor div} the Knutson Index divides the lowest common of all simple representations of all subgroups of $S_3$. As these representations are one or two dimensional, the Knutson Index of $G$ divides $2$ and we conclude that $\mathcal{K}(G)$ is $1$ or $2$. For instance, if $G = \text{Spin}_8 \times S_3$ we know that $\mathcal{K}(G)=1$.
\end{ex}

\begin{ex}
Recall from Example \ref{ex ss} the groups with $G_0 = (SL_2)^3$ and $H = S_4$. By the proof of Corollary \ref{cor div} the Knutson Index divides the lowest common of all simple projective representations of all subgroups of $S_4$. As these representations are one, two, three and four dimensional, the Knutson Index of $G$ divides $12$ and we conclude that $\mathcal{K}(G)$ is $1, 2, 3, 4, 6$ or $12$. For instance, if $G = (SL_2)^3 \times S_4$ we know that $\mathcal{K}(G)=1$.
\end{ex}


\bibstyle{numeric}
\printbibliography

@inproceedings {Sp,
  title={Reductive Groups},
  author={Springer, Tonny},
  booktitle = {Automorphic Forms, Representations and $L$-Functions},
  volume={1},
  year={1979},
  pages = {3-27},
  publisher={American Mathematical Society}
}

@article {Sp2,
  author = {Springer, Tonny},
  title = {Twisted conjugacy in simply connected groups},
  journal =  {Transformation Groups},
  volume = {11},
  year = {2006},
  number = {3},
  pages = {539--545},
}

@article{Se,
  title={Equivariant {K}-theory},
  author={Segal, Graeme},
  journal={Publications Math{\'e}matiques De L'I.H.{\'E}.S.},
  volume={34},
  pages={129--151},
  year={1968}
}

@article{DM,
  title={Groupes r{\'e}ductifs non connexes},
  author={Digne, Fran{\c{c}}ois and Michel, Jean},
  journal={Annales scientifiques de l'Ecole normale sup{\'e}rieure},
  volume={27},
  number={3},
  pages={345--406},
  year={1994}
}

@article{AHR,
  title={Representation theory of disconnected reductive groups},
  author={Achar, Pramod N. and Hardesty, William D. and Riche, Simon},
  journal={Documenta Mathematica},
  volume={25},
  pages={2149--2177},
  year={2020}
}

@article {EM,
  author = {Eilenberg, Samuel and MacLane, Saunders},
  title = {Cohomology theory in abstract groups. {II}. {G}roup extensions
              with a non-{A}belian kernel},
  journal = {Annals of Mathematics},
  volume = {48},
  year = {1947},
  pages = {326--341}
}

@article{BS,
  title={Th{\'e}oremes de finitude en cohomologie galoisienne},
  author={Borel, Armand and Serre, Jean-Pierre},
  journal={Commentarii Mathematici Helvetici},
  volume={39},
  number={1},
  pages={111--164},
  year={1964}
}

@article{GV,
  title={Disconnected reductive groups},
  author={Gaetz, Marisa and Vogan Jr, David A.},
  journal={Journal of Lie Theory},
  volume={34},
  number={2},
  pages={469--480},
  year={2024}
}

@article{Y,
  title={Isomorphisms of group extensions},
  author={Yang, Kung-Wei},
  journal={Pacific Journal of Mathematics},
  volume={50},
  number={1},
  pages={299--304},
  year={1974}
}

@article{M,
  title={The Knutson Index of the representation ring},
  author={Mart{\'\i}n Duro, Diego},
  journal={Journal of Algebra},
  volume={659},
  pages={516--541},
  year={2024}
}

@article{JMR,
  title={Global Representation Ring and Knutson Index},
  author={Johnston, Dylan and Mart{\'i}n Duro, Diego and Rumynin, Dmitriy},
  journal={Communications in Algebra},
  volume={53},
  number={11},
  pages={4793--4804},
  year={2025}
}

@article{W,
  title={Some applications of group obstructions},
  author={Wu, Yel-Chiang},
  journal={Journal of the Australian Mathematical Society},
  volume={32},
  number={2},
  pages={178--186},
  year={1982}
}

@article {Lus,
  author = {Lusztig, George},
  title= {Character sheaves on disconnected groups. {I}},
  journal = {Representation Theory. An Electronic Journal of the American
              Mathematical Society},
  volume = {7},
  pages={374--403},
  year = {2003}
}

@article{With,
  title={The ring of equivariant vector bundles on finite sets},
  author={Witherspoon, Sarah},
  journal={Journal of Algebra},
  volume={175},
  number={1},
  pages={274--286},
  year={1995}
}

@book {Robinson,
  author = {Robinson, Derek J. S.},
  title = {A course in the theory of groups},
  publisher = {Springer-Verlag, New York},
  year = {1996}
}

@book{J,
  title={Representations of Algebraic Groups},
  author={Jantzen, Jens C.},
  year={2003},
  publisher={American Mathematical Society}
}

@book {Mackey,
  author = {Mackey, George W.},
  title = {The theory of unitary group representations},
  publisher = {University of Chicago Press, Chicago, Ill.-London},
  year = {1976},
}

@book{B,
  title={Cohomology of groups},
  author={Brown, Kenneth S.},
  year={2012},
  publisher={Springer Science \& Business Media}
}

@book {Spa,
  author = {Spaltenstein, Nicolas},
  title = {Classes unipotentes et sous-groupes de {B}orel},
  publisher = {Springer-Verlag},
  year = {1982},
}

@book{FH,
  title={Representation theory: a first course},
  author={Fulton, William and Harris, Joe},
  volume={129},
  year={2013},
  publisher={Springer Science \& Business Media}
}

\end{document}